\documentclass[a4paper]{amsart}
\usepackage{eucal}
\usepackage{enumerate}

\theoremstyle{plain}
\newtheorem{thm}{Theorem}[subsection]
\newtheorem*{thm*}{Theorem}

\newtheorem{prop}{Proposition}[subsection]
\newtheorem*{prop*}{Proposition}
\newtheorem{lemma}{Lemma}[subsection]
\newtheorem*{lemma*}{Lemma}
\newtheorem{coro}{Corollary}[subsection]
\newtheorem*{coro*}{Corollary}
\newtheorem{conj}{Conjecture}[subsection]

\theoremstyle{definition}
\newtheorem{dfn}{Definition}[subsection]
\newtheorem*{dfn*}{Definition}
\newtheorem{rem}{Remark}[subsection]
\newtheorem*{rem*}{Remark} 

\newtheorem*{ex*}{Example}
\newtheorem{hyps}{Hypotheses}[subsection]
\newtheorem{hyp}{Hypothesis}[subsection]

\newcommand{\A}{\mathbb A}
\newcommand{\C}{\mathbb{C}}
\newcommand{\R}{\mathbb R}
\newcommand{\Z}{\mathbb Z}
\newcommand{\Q}{\mathbb Q}

\newcommand{\Gal}{\operatorname{Gal}}
\newcommand{\Hom}{\operatorname{Hom}}
\newcommand{\Aut}{\operatorname{Aut}}

\newcommand{\Res}{\operatorname{Res}}
\newcommand{\GL}{\operatorname{GL}}

\newcommand{\gr}{\operatorname{gr}}

\newcommand{\Ind}{\operatorname{Ind}}
\newcommand{\St}{\operatorname{St}}

\newcommand{\hol}{\operatorname{hol}}

\newcommand{\dR}{\operatorname{dR}}

\newcommand{\mot}{\operatorname{mot}}
\newcommand{\Coh}{\operatorname{Coh}}

\numberwithin{equation}{subsection}

\begin{document} 

\title[Critical values of $L$-functions of potentially automorphic motives]{On critical values of $L$-functions of potentially automorphic motives}
\author{Daniel Barrera Salazar}
\address{Centre de Recherches Math\'ematiques, Universit\'e de Montr\'eal, Pavillon Andr\'e-Aisenstadt, 2920, Chemin de la tour, Montr\'eal (Qu\'ebec), H3T 1J4, Canada}
\email{danielbarreras@hotmail.com}
\author{Lucio Guerberoff}
\address{Department of Mathematics, University College London, 25 Gordon Street, London WC1H 0AY, UK}
\email{l.guerberoff@ucl.ac.uk}

\begin{abstract} In this paper we prove a version of Deligne's conjecture for potentially automorphic motives, twisted by certain algebraic Hecke characters. The Hecke characters are chosen in such a way that we can use automorphic methods in the context of totally definite unitary groups. 
\end{abstract}

\maketitle

\section{Introduction}\label{sec:intro} 

The goal of this paper is to prove results on Deligne's conjecture for potentially automorphic motives, twisted by certain algebraic Hecke characters. Let $K$ be a totally real number field and $L/K$ be a CM extension. We refer the reader to Section \ref{sec:motivic} for an overview of motives and realizations. For a realization $M$ over $K$ and an algebraic Hecke character $\chi$ of $L$, we let $M(\chi)$ denote the tensor product $M\otimes\Res_{L/K}[\chi]$, where $[\chi]$ is the CM motive over $L$ attached to $\chi$. A realizations $M$ is automorphic if it looks like the conjectural motive attached to a self-dual, cohomological, cuspidal automorphic representation $\Pi$ of $\GL_{n}(\A_{K})$; for more details, see Subsection \ref{ssec:automrealiz}. A realization $M$ is potentially automorphic if there exists a finite, Galois, totally real extension $K'/K$ such that $M_{K'}$ is automorphic. Our main result (see Theorem \ref{potentially}) is the following. In the statement, $\tilde L$ is a certain finite Galois extension of $L$, which equals the Galois closure of $L$ if $M$ is automorphic. We refer the reader to the main text for more details and for any unexplained notation. 

\begin{thm*} Let $M$ be a potentially automorphic realization over $K$, with coefficients in $E$, of rank $n$, satisfying condition $(3)$ in Theorem \ref{potentially}. Let $\psi$ be a critical algebraic Hecke character of $L$ of infinity type $(m_{\tau})_{\tau \in \Hom(L, \C)}$ and weight $w$. Assume that:
	\begin{itemize}
		\item[(1)] either $n$ is even, or $n$ is odd and Hypothesis \ref{hypo:delta} is satisfied for $M$, and
		\item[(2)] $\mid m_{\tau}- m_{\overline{\tau}}\mid > \max\left\lbrace n- p_{n}(\sigma, 1)\right\rbrace_{\sigma \in \Hom(K, \C)}$ for any $\tau \in \Hom(L, \C)$.
	\end{itemize}
	Then the weak form of Deligne's conjecture up to $\Q(\psi)\tilde L$-factors is true for all critical integers $k>w+n$ of $M(\chi)$ and the embedding $1\in\Hom(E(\chi),\C)$. That is, for such integers $k$, we have
	\[ \frac{L(M(\chi),k)_{1}}{c^{+}(M(\chi)(k))_{1}}\in{(E\Q(\psi)\tilde L)^{\times}}. \]
\end{thm*}
Here $p_{1}(\sigma, 1)>\dots>p_{n}(\sigma,1)$ are the Hodge numbers of $M$ (see Subsection \ref{ssec:polregmot}) and $\chi$ is an algebraic Hecke character constructed from $\psi$ as follows: there exists a finite order character $\psi_{0}$ of $\A_{K}^{\times}/K^{\times}$ such that $\psi|_{\A_{K}^{\times}}=\psi_{0}\|\cdot\|^{-w}$, and we put $\chi=\psi^{2}(\psi_{0}\circ N_{L/K})^{-1}$. We stress that, given $M$, there always exists algebraic Hecke characters and integers $k>w+n$ critical for $M(\chi)$ as in the statement. 

The proof of our theorem works as follows. First, suppose that $M$ is automorphic. Hypothesis (2) in the theorem allows us, using the results proved by one of the authors (\cite{guerbperiods}) generalizing earlier results of Harris (\cite{harriscrelle}) to totally real fields, to write the value of the $L$-function of $M(\chi)$ at $k$ in terms of a CM period attached to $\chi$. The way this works is by using descent to a totally definite unitary group $G$ and expressing this $L$-function as that of a cohomological automorphic representation on $G$. The corresponding critical value is expressed in terms of a CM period and a Petersson norm, which in our case turns out to be algebraic. On the other hand, the motivic computations of \cite{guerbperiods} and Hypothesis (2) allow us to express the Deligne period of $M(\chi)(k)$ in terms of a CM period, which turns out to match the previous one. Thus, we prove in this way that the value $L(M(\chi),k)_{1}$ at any critical integer $k$ of $M(\chi)$ which satisfies $k>w+n$ equals the Deligne period $c^{+}(M(\chi)(k))_{1}$, up to multiplication by an element of $E\Q(\psi)\tilde L$. The requirement that $k>w+n$ is necessary for the automorphic methods to work. In the same vein, automorphic methods only allow us to consider multiples by $E\Q(\psi)\tilde L$. This consideration is included in Conjecture \ref{conj:weakdeligne}, which is what we call the weak form of Deligne's conjecture.

When $M$ is a potentially automorphic realization over $K$, we use Brauer's induction and solvable base change for $\GL_{n}$, as developed in the theory of Arthur-Clozel (\cite{arthurclozel}). We prove the theorem for $M$ by using the previous automorphic case and further compatibilities between the CM periods that appear, which are a consequence of Deligne's conjecture for algebraic Hecke characters, proved in this case by Blasius (\cite{blasiusannals}).

Let us say a few words about the hypotheses of the theorem. As we mentioned before, a potentially automorphic realization $M$ is one such that it becomes automorphic after extension of scalars to a totally real, Galois extension $K'/K$. A number of techniques is available to prove that certain motives, or Galois representations, are potentially automorphic. We refer the reader to \cite{blgt} for some general results in this direction. From now on, for simplicity, assume that $M$ is already automorphic. In the statement of the theorem, we require $\psi$ to be critical. This means that $m_{\tau}\neq m_{\bar\tau}$ for any $\tau\in\Hom(L,\C)$. This is needed for the motivic computations in \cite{guerbperiods}. Regarding Hypothesis (1) in the theorem, it is needed to express the full period $\delta(M)$ in a convenient way. This period appears in the motivic computations that give rise to the expression of $c^{+}(M(\chi)(k))$. Automorphic realizations are endowed with a polarization, which can be used to show that $\delta(M)$ can be replaced by $(2\pi i)^{-[K:\Q]n(n-1)/2}$ when $n$ is even. Hypothesis (1) says that that when $n$ is odd, we can also replace $\delta(M)$ by the same power of $2\pi i$. This can be shown to be true if one assumes the much stronger Tate conjecture. Hypothesis (2), as we explained above, is made so that $M(\chi)$ has critical values and so that they can be expressed essentially as a CM period (in the terminology of \cite{harriscrelle}, $\chi$ belongs to the $n$-th critical interval of $M$). It is essential to the method, and a modification of it, still assuming that $M(\chi)$ has critical values, would entail the appearance of additional quadratic periods in the formulas, which are harder to relate to critical values of $L$-functions. Hypothesis (3) says that $\Pi$ has a descent to a totally definite unitary group satisfying a number of conditions. This is expected to hold in our setting, and many, if not most, cases have been already proved (\cite{labesse}, \cite{mok}, \cite{kmsw}). Combined with Hypothesis (2), this implies that we can express the $L$-function as that of a cohomological automorphic representation on a totally definite unitary group. 

While doing the motivic computations behind the expression of $c^{+}(M(\chi)(k))$, we draw some consequences of the general formula proved in \cite{guerbperiods} for the case of arbitrary critical intervals. More precisely, if $M$ is a regular, polarized realization, then after fixing embeddings of the coefficient fields into $\C$, we construct certain algebraic Hecke characters $\chi$ with the property that $M(\chi)$ has critical values which can be expressed in terms of CM periods and additional quadratic periods $Q_{j}(M)$ attached to $M$. The set of quadratic periods appearing in the expression depends on the critical interval of $\chi$, which can be arbitrarily prescribed. Combining the formulas for different characters $\chi$ gives an expression of the quadratic periods in terms of quotients of Deligne periods of various twists $M(\chi)$ and CM periods (see Proposition \ref{prop:Qjota} and Corollary \ref{coro:Qjota}). Assuming Deligne's conjecture for these motives, we can then express the quadratic periods in terms of critical values of $L$-functions and certain Gauss sums (Proposition \ref{qj formula}). These expressions should be helpful in certain applications, and we plan to exploit this in a future work in relation to $p$-adic interpolation and $p$-adic $L$-functions.

To finish this introduction we would like to say a few words about the background and motivation for this work. The study of values of $L$-functions at integers is a subject with a long history starting with Euler. Based on experimental results, in $1979$ Deligne proposed a general conjecture relating the values of motivic $L$-functions at certain integer points to periods of integrals. Arguably the most useful way to approach this and other conjectures on critical values of $L$-functions is by using automorphic methods. In this direction, let us mention Blasius's result (\cite{blasiusannals}), which proves the conjecture for the motives attached to algebraic Hecke characters of CM fields (see also \cite{hardersch}). Let us also mention Shimura's works on critical values of $L$-functions in the case of modular forms and Hilbert modular forms (see for instance \cite{shimurazeta} and \cite{shimuraduke}), based on the doubling method. Following these works, Harris in \cite{harriscrelle} treated the case of the polarized regular motives over $\Q$ coming from automorphic representations of $\GL_{n}(\A_{\Q})$. Finally, in \cite{guerbperiods}, the author generalized Harris's results over arbitrary totally real number fields. The present work is based on \cite{guerbperiods}.

\subsection{Organization of the paper} Here is the outline of the paper. Section \ref{sec:motivic} is devoted to the motivic considerations of this work. The main result of this section is a formula for the Deligne period of $M(\chi)(k)$ in terms of CM periods and powers of $2\pi i$, based on \cite{guerbperiods} (Proposition \ref{prop:exprc+k}). As mentioned above, with an eye towards future applications, we also include a formula for the quadratic periods of a regular, polarized realization in terms of quotients of various twists of $M$ and CM periods (Proposition \ref{prop:Qjota} and its corollary). In Section \ref{sec:automorphic} we specialize to the case of totally definite unitary groups the results of \cite{guerbperiods} on critical values of automorphic $L$-functions (Proposition \ref{prop:maincritical}). In Section \ref{sec:deligne} we combine the previous sections and prove the main result of this paper (Theorems \ref{thm:main} and \ref{potentially} for automorphic and potentially automorphic realizations respectively). Finally, in Section \ref{sec:quadratic}, which is more speculative, we write down an expression of the quadratic periods in terms of critical values of $L$-functions (Proposition \ref{qj formula}), which follows from assuming Deligne's conjecture for the corresponding motives.

\subsection*{Acknowledgements} The second author wants to thank Adrian Iovita and the Centre de Recherches Math\'ematiques in Montr\'eal for supporting his visit in Spring 2016, during which this work was carried out. 

\subsection*{Notation and conventions}
We fix an algebraic closure $\C$ of $\R$, a choice of $i=\sqrt{-1}$, and we let $\bar\Q$ denote the algebraic closure of $\Q$ in $\C$. We let $c\in\Gal(\C/\R)$ denote complex conjugation on $\C$, and we use the same letter to denote its restriction to $\bar\Q$. We write $c(z)=\bar z$ for $z\in\C$.

For a number field $K$, we let $\A_{K}$ and $\A_{K,f}$ denote the rings of ad\`eles and finite ad\`eles of $K$ respectively. When $K=\Q$, we write $\A=\A_{\Q}$ and $\A_{f}=\A_{\Q,f}$. After fixing an algebraic closure $\bar K$ of $K$, we denote $\Gamma_{K}= \Gal(\bar K/K)$. 

A CM field $L$ is a totally imaginary quadratic extension of a totally real field $K$. A CM type $\Phi$ for $L/K$ is a subset $\Phi\subset\Hom(L,\C)$ such that $\Hom(L,\C)=\Phi\coprod c\Phi$ (equivalently, a choice of one of the two possible extensions to $L$ of each embedding of $K$ in $\C$).

All vector spaces over fields will be finite-dimensional except otherwise stated. A tensor product without a subscript between $\Q$-vector spaces will always mean tensor product over $\Q$. For any number field $K$, we denote by $J_{K}=\Hom(K,\C)$. For $\sigma\in J_{K}$, we let $\bar\sigma=c\sigma$. Let $E$ and $K$ be number fields, and $\sigma\in J_K$. If $\alpha,\beta\in E\otimes \C$, we write $\alpha\sim_{E\otimes K,\sigma}\beta$ if either $\beta=0$ or if $\beta\in(E\otimes \C)^{\times}$ and $\alpha/\beta\in(E\otimes\sigma(K))^{\times}$. There is a natural isomorphism
$E\otimes\C\simeq\prod_{\varphi\in J_E}\C$ given by $e\otimes z\mapsto(\varphi(e)z)_{\varphi}$ for $e\in E$ and $z\in\C$. Under this identification, we denote an element $\alpha\in E\otimes\C$ by $(\alpha_{\varphi})_{\varphi\in J_E}$. When $K$ is given from the context as a subfield of $\C$, we write $\sim_{E\otimes K}$ for $\sim_{E\otimes K,1}$, where $1:K\hookrightarrow\C$ is the given embedding. 

Suppose that $\mathbf{r}=(r_{\varphi})_{\varphi\in J_E}$ is a tuple of nonnegative integers. Given $Q_{1},\dots,Q_{n}$ in $E\otimes\C$ (with $n\geq r_{\varphi}$ for all $\varphi$), we denote by
\[ \prod_{j=1}^{\mathbf{r}}Q_{j}\in E\otimes\C \]
the element whose $\varphi$-th coordinate is $\prod_{j=1}^{r_{\varphi}}Q_{j,\varphi}$

\section{Motives and periods}\label{sec:motivic}

\subsection{Algebraic Hecke characters} Let $L$ be a number field. For a place $v$ of $L$, we denote by $L_{v}$ the corresponding completion, and by $(L_{v}^{\times})^{+}$ the connected component of the identity in $L_{v}^{\times}$. An algebraic Hecke character of $L$ is a continuous character $\chi:L^{\times}\backslash\A_{L}^{\times}\to\C^{\times}$, with the property that for each embedding $\tau\in J_{L}$, there exists an integer $n_{\tau}$ such that if $v$ is the archimedean place of $L$ induced by $\tau$, then for every $x\in(L_{v}^{\times})^{+}$, 
\[ \chi(x)=\left\{\begin{array}{lll}\tau(x)^{-n_{\tau}} & \text{if} & v \text{ is real},\\
                                                   \tau(x)^{-n_{\tau}}\bar\tau(x)^{-n_{\bar\tau}} & \text{if} & v\text{ is complex}.\end{array}\right.\]
The integer $n_{\tau}+n_{\bar\tau}$ is independent of $\tau$, and it's called the weight $w(\chi)$ of $\chi$. The tuple $(n_{\tau})_{\tau\in J_{L}}$ is called the infinity type of $\chi$. Let $\Q(\chi)\subset\C$ be the field generated by the values of $\chi$ on the finite id\`eles $\A_{L,f}^{\times}$. Then $\Q(\chi)$ is either $\Q$ or a CM field. 

From now on, assume that $L$ is a CM field, which is the case that we will be interested in. Let $K$ denote the maximal totally real subfield of $L$. The restriction of $\chi$ to $K$ must necessarily be of the form
\[ \chi|_{\A_{K}^{\times}}=\chi_{0}\|\cdot\|^{-w(\chi)},\]
where $\|\cdot\|$ is the id\`elic norm on $\A_{K}^{\times}$, and $\chi_{0}$ is of finite order (see \cite{schbook}, Chapter 0).

Let $\Phi$ be a CM type for $L/K$, that is, $\Phi\subset J_{L}$ consists of a choice of one of the two possible extensions of each $\sigma\in J_{K}$ to $L$. 

\begin{prop}\label{prop:existechi} Let $(a_{\tau})_{\tau\in\Phi}$ be a tuple of integers such that
\[ a_{\tau}\equiv a_{\tau'}(2) \]
for every $\tau,\tau'\in\Phi$. Then there exists an algebraic Hecke character $\chi$ of $L$ of infinity type $(n_{\tau})_{\tau\in J_{L}}$ such that
\[ n_{\tau}-n_{\bar\tau}=a_{\tau}\quad(\tau\in\Phi).\]
Moreover, if $w_{0}\in\Z$ satisfies that $w_{0}\equiv a_{\tau}(2)$ for one (or every) $\tau\in J_{L}$, then $\chi$ can be taken to have weight $w_{0}$.
\begin{proof}
As in \cite{schbook}, Chaper 0, any tuple of integers $(n_{\tau})_{\tau\in J_{L}}$ with the property that $n_{\tau}+n_{\bar\tau}$ is independent of $\tau$ is the infinity type of an algebraic Hecke character $\chi$ of $L$. To arrive at the conditions of the proposition, choose an arbitrary $w_{0}$, with the same parity as the $a_{\tau}$, and take the tuple $(n_{\tau})_{\tau\in\Phi}$ as
\[ n_{\tau}=\frac{w_{0}+a_{\tau}}{2},\quad\tau\in\Phi,\]
\[ n_{\tau}=\frac{w_{0}-a_{\bar\tau}}{2},\quad\tau\not\in\Phi.\]
\end{proof}
\end{prop}

\subsection{Polarized regular motives}\label{ssec:polregmot} In this subsection, we recall the main result of \cite{guerbperiods}, Section 2. Let $K$ be a totally real number field, $L/K$ a CM extension and $E$ any number field. We are fixing throughout an algebraic closure $\bar K$ of $K$, and we take $L$ inside $\bar K$. Let $M$ be a realization over $K$ of rank $n$ with coefficients in $E$, pure of weight $w(M)$. We refer to \cite{guerbperiods} for a basic overview of realizations. Here we stress that $M$ consists of a collection of vector spaces and comparison isomorphisms, the interesting cases being given by collections of realizations coming from motives for absolute Hodge cycles over $K$, as in Deligne's article \cite{deligne}. For each $\sigma\in J_{K}$, we can define the period $\delta_{\sigma}(M)$. These are elements of $(E\otimes\C)^{\times}$, well defined up to multiplication by an element of $(E\otimes\sigma(K))^{\times}$. We let $\delta(M)=\delta_{1}(\Res_{K/\Q}M)$, with $1\in J_{\Q}$ being the unique embedding of $\Q$. 

We say that $M$ is {\em special} if, for every $\sigma\in J_{K}$, the action of the Frobenius automorphism $F_{\sigma}$ on the Hodge component $M_{\sigma}^{w(M)/2,w(M)/2}$ is given by a scalar $\varepsilon=\pm1$ (independent of $\sigma$). Here the $E$-vector space $M_{\sigma}$ is the Betti realization of $M$ attached to the embedding $\sigma$. Under the condition that $M$ is special, we can define the Deligne $\sigma$-periods $c_{\sigma}^{\pm}(M)$, again elements of $(E\otimes\C)^{\times}$ well defined modulo $(E\otimes\sigma(K))^{\times}$. We let $c^{\pm}(M)=c^{\pm}(\Res_{K/\Q}M)$. We define $n^{\pm}=\dim_{E}M_{\sigma}^{\pm}$, where
\[ M_{\sigma}^{\pm}=\{x\in M_{\sigma}:F_{\sigma}(x)=\pm x\}.\]
This is independent of $\sigma$. The following factorization formula is proved in \cite{yoshida} or \cite{panch} (we also include a similar formula for the $\delta$'s):
\begin{align}\label{form:yoshida1} c^{\pm}(M)\sim_{E} D_{K}^{n^{\pm}/2}\prod_{\sigma}c^{\pm}_{\sigma}(M),\\
\label{form:yoshida2} \delta(M)\sim_{E} D_{K}^{n/2}\prod_{\sigma}\delta_{\sigma}(M).\end{align}
Here $D_{K}$ is the discriminant of $K$. Note that $D_{K}$ is a positive integer such that $D_{K}^{1/2}\in K^{\Gal}$, where $K^{\Gal}\subset\bar\Q$ is the Galois closure of $K$ in $\bar\Q$. 

Under the assumption that the system of $\lambda$-adic representations $(M_{\lambda})_{\lambda}$ is strictly compatible, we can define the $L$-function of $M$, $L^{*}(M,s)=(L(M,s)_{\varphi})_{\varphi\in J_{E}}\in E\otimes\C$ (see \cite{deligne}). We will always assume this to be the case. We also refer the reader to {\em op. cit.} for the definition of critical integers. Deligne's conjecture is the statement saying that if $M$ is critical, meaning that $0$ is a critical integer, then
\[ \frac{L^{*}(M,0)}{c^{+}(M)}\in E.\]
We will later prove, for certain realizations $M$, a weaker version of this conjecture which we state as follows. Note that if $k$ is a critical integer for $M$, then the $k$-th Tate twist $M(k)$ is critical, and $L^{*}(M(k),0)=L^{*}(M,k)$. Also, $c^{+}(M(k))\sim(2\pi i)^{[K:\Q]n^{\pm}k}c^{\pm}(M)$, where $\pm=(-1)^{k}$.

\begin{conj}\label{conj:weakdeligne} Let $F\subset\C$ be a number field and $\varphi\in J_{E}$. Let $k$ be a critical integer of $M$. Then
	\[ \frac{L(M,k)_{\varphi}}{c^{+}(M(k))_{\varphi}}\in F\varphi(E).\]
\end{conj}
We will usually refer to Conjecture \ref{conj:weakdeligne} for $(M,\varphi,F,k)$ as the {\em the weak form of Deligne's conjecture} up to $F$-factors for $\varphi$ and the critical integer $k$.

We say that the realization $M$ is {\em regular} if, for every $\sigma\in J_{K}$ and $\varphi\in J_{E}$, the spaces $M_{\sigma}^{pq}(\varphi)$ have dimension at most $1$ over $\C$. Associated to the pair $(\sigma,\varphi)$, there is a sequence of integers
\[ p_{1}(\sigma,\varphi)>\dots>p_{n}(\sigma,\varphi) \] 
with the property that $M_{\sigma}^{pq}(\varphi)\neq 0$ if and only if $p=p_{i}(\sigma,\varphi)$ for some $i=1,\dots,n$. We let $q_{i}(\sigma,\varphi)=w(M)-p_{i}(\sigma,\varphi)$, which in fact equals $p_{n+1-i}(\sigma,\varphi)$. We let $p_{0}(\sigma,\varphi)=+\infty$ and $p_{n+1}(\sigma,\varphi)=-\infty$. Note that if $n=2k-1$ is odd, then $p_{k}(\sigma,\varphi)=p_{n+1-k}(\sigma,\varphi)$, which implies that $w(M)$ is even. In particular, $w(M)n$ is even in all cases.

Let $\chi:L^{\times}\backslash\A_{L}^{\times}\to\C^{\times}$ be an algebraic Hecke character of $L$ of infinity type $(n_{\tau})_{\tau\in J_{L}}$. Attached to $\chi$ is a CM motive over $L$ with coefficients in $\Q(\chi)$, which we denote by $[\chi]$ (this is denoted by $M(\chi)$ in \cite{guerbperiods}). See \cite{schbook} for the construction of $[\chi]$. We let $\Res_{L/K}[\chi]$ be the motive over $K$ obtained by restriction of scalars from $L$ to $K$ of $[\chi]$. We say that $\chi$ is {\em critical} if $n_{\tau}\neq n_{\bar\tau}$ for every $\tau\in J_{L}$. In this case, $\Res_{L/K}[\chi]$, a realization of rank $2$, is regular in the sense defined above, and we denote the corresponding Hodge numbers by 
\[ p_{1}^{\chi}(\sigma,\rho)>p_{2}^{\chi}(\sigma,\rho),\]
where $\rho\in J_{\Q(\chi)}$. If we let $n(\tau,\rho)=n_{\tilde\rho^{-1}\tau}$, where $\tilde\rho:\C\to\C$ is an extension of $\rho$ to $\C$, then 
\[ \{p_{1}^{\chi}(\sigma,\rho),p_{2}^{\chi}(\sigma,\rho)\}=\{n(\tau,\rho),n(\bar\tau,\rho)\},\]
where $\tau$ and $\bar\tau$ are the two embeddings of $L$ extending $\sigma$. For $\chi$ critical, we define
\[ t_{\sigma,\rho}(\chi)=p_{1}^{\chi}(\sigma,\rho)-p_{2}^{\chi}(\sigma,\rho).\]
We let $c_{\sigma}^{\pm}(\chi)=c_{\sigma}^{\pm}(\Res_{L/K}[\chi])$. For any $\tau\in J_{L}$, let $e_{\tau}=(e_{\tau,\rho})_{\rho\in J_{\Q(\chi)}}\in(\Q(\chi)\otimes\C)^{\times}$ be the element whose $\rho$-coordinate is $e_{\tau,\rho}=1$ if $n(\tau,\rho)>n(\bar\tau,\rho)$ and $e_{\tau,\rho}=-1$ if $n(\tau,\rho)<n(\bar\tau,\rho)$. Recall that the restriction of $\chi$ to $\A_{K}^{\times}$ can be written as
\[ \chi|_{\A_{K}^{\times}}=\chi_{0}\|\cdot\|^{-w(\chi)},\]
where $\chi_{0}$ is a finite order character. We let $\varepsilon_{L}$ denote the finite order character of $K^{\times}\setminus\A_{K}^{\times}$ corresponding under class field theory to the quadratic character of $\Gal(\bar K/K)$ associated with the extension $L/K$. We let $[\chi_{0}\varepsilon_{L}]$ denote the Artin motive of rank $1$ over $K$, with coefficients in $\Q(\chi)$, attached to the finite order character $\chi_{0}\varepsilon_{L}$. For each $\sigma\in J_{K}$, we let
\[ G_{\sigma}(\chi)=\delta_{\sigma}[\chi_{0}\varepsilon_{L}].\]

We let $M(\chi)=M\otimes \Res_{L/K}[\chi]$ be the tensor product of the realizations $M$ and $\Res_{L/K}[\chi]$. This tensor product is over $\Q$, meaning that it has coefficients in $E(\chi)=E\otimes\Q(\chi)$. To be more precise, $E(\chi)$ is a product of number fields, and $M(\chi)$ is a collection of realizations with coefficients in each of these fields. For simplicity in most of what follows, we will simply assume that $E(\chi)$ is a field. Suppose that $M$ is regular and $\chi$ is critical. It's shown in \cite{guerbperiods} (Proposition 2.5.1) that $M(\chi)$ has critical values if and only if, for every $\sigma\in J_{K}$ and every $\rho\in J_{\Q(\chi)}$, 
\[ t_{\sigma,\rho}(\chi)\neq w(M)-2p_{i}(\sigma,\varphi) \]
for any $i=1,\dots,n$ and any $\varphi\in J_{E}$. Assuming this is the case, we can find integers $r_{\sigma,\varphi,\rho}(\chi)\in\{0,\dots,n\}$ such that
\[ w(M)-2p_{r_{\sigma,\varphi,\rho}(\chi)}(\sigma,\varphi)<t_{\sigma,\rho}(\chi)<w(M)-2p_{r_{\sigma,\varphi,\rho}(\chi)+1}(\sigma,\varphi).\]
We stress here that if $K\neq\Q$, the numbers $p_{i}(\sigma,\varphi)$ depend in general on the choice of the embeddings. Moreover, the integers $r_{\sigma,\varphi,\rho}(\chi)$ also depend on the choice of the three embeddings $\sigma$, $\varphi$ and $\rho$.

We say that $M$ is {\em polarized} if there exists a non-degenerate morphism of realizations 
\[ \langle,\rangle:M\otimes_{E}M\to E(-w(M)). \]
We also impose the condition that $\langle,\rangle$ is symmetric if $w(M)$ is even, and alternated if $w(M)$ is odd. We refer to Subsection 2.3 of \cite{guerbperiods} for the definition of the quadratic periods $Q_{j,\sigma}\in(E\otimes\C)^{\times}$ attached to a polarized realization $M$. Assume from now on that $M$ is a regular, polarized, special realization, pure of weight $w(M)$ and rank $n$. 

Let $\chi$ be an critical algebraic Hecke character of infinity type $(n_{\tau})_{\tau\in J_{L}}$ and weight $w(\chi)$. We suppose that $M(\chi)$ has critical values, and we let $s_{\sigma,\varphi,\rho}(\chi)=n-r_{\sigma,\varphi,\rho}(\chi)$, $\mathbf{r}_{\sigma}=(r_{\sigma,\varphi,\rho}(\chi))_{\varphi,\rho}$ and $\mathbf{s}_{\sigma}=(s_{\sigma,\varphi,\rho}(\chi))_{\varphi,\rho}$. We let 
\[ a_{\sigma}^{\pm}(\chi)=(2\pi i)^{w(\chi)}G_{\sigma}(\chi)^{-1}c_{\sigma}^{\mp}(\chi) \]
and
\[ Q_{\sigma}(\chi)=(2\pi i)^{w(\chi)}G_{\sigma}(\chi)^{-2}e_{\tau}c_{\sigma}^{+}(\chi)^{2}.\]
Here $\tau$ is any of the two embeddings of $L$ extending $\sigma$. Since $e_{\tau}=-e_{\bar\tau}$, this definition makes sense in $(\Q(\chi)\otimes\C)^{\times}$ modulo $(\Q(\chi)\otimes\sigma(K))^{\times}$. 

\begin{rem} These quantities are defined more conceptually in \cite{guerbperiods}. The formulas above are obtained (up to multiples in $(\Q(\chi)\otimes \sigma(K))^{\times}$) in (2.4.3) and Proposition 2.5.2 of {\em op. cit.}.
\end{rem}

Theorem 2.5.1 of \cite{guerbperiods} (see also Proposition 1.7.6 of \cite{harriscrelle} when $K=\Q$) says that 
\begin{equation}\label{eqn:formppal} c_{\sigma}^{+}(M(\chi))\sim(2\pi i)^{-\lceil n/2\rceil w(\chi)}G_{\sigma}(\chi)^{\mathbf{r}_{\sigma}}\delta_{\sigma}(M)a_{\sigma}^{*}(\chi)Q_{\sigma}(\chi)^{\mathbf{r}_{\sigma}-\lceil n/2\rceil}\prod_{j=1}^{\mathbf{s}_{\sigma}}Q_{j,\sigma}, \end{equation}
where $a_{\sigma}^{*}(\chi)=1$ if $n$ is even, and $a_{\sigma}^{*}(\chi)=a_{\sigma}^{\pm}(\chi)$ if $n$ is odd, with $\pm=-$ if $n^{+}>n^{-}$ and $\pm=+$ if $n^{-}>n^{+}$. In the formula, $\sim$ means $\sim_{E(\chi)\otimes K,\sigma}$. In particular, if we look at the coordinates of the elements corresponding to an embedding $\varphi\in J_{E}$ and the embedding $1\in J_{\Q(\chi)}$, we get
\begin{equation}\label{eqn:formppal2} c_{\sigma}^{+}(M(\chi))_{\varphi,1}\sim\end{equation}\[(2\pi i)^{-\lceil n/2\rceil w(\chi)}G_{\sigma}(\chi)^{r_{\sigma,\varphi,1}(\chi)}_{1}\delta_{\sigma}(M)_{\varphi}a_{\sigma}^{*}(\chi)_{1}Q_{\sigma}(\chi)^{r_{\sigma,\varphi,1}(\chi)-\lceil n/2\rceil}_{1}\prod_{j=1}^{s_{\sigma,\varphi,1}(\chi)}Q_{j,\sigma,\varphi}, \]
where now both sides are complex numbers and $\sim$ stands for $\sim_{\varphi(E)\Q(\chi)\sigma(K)}$. When $E$ is given as a subfield of $\C$ with a given embedding $1\in J_{E}$, we write $c^{+}(M(\chi))_{1}$ for $c^{+}(M(\chi))_{1,1}$, and similarly for any other element of $E(\chi)\otimes\C$.

\subsection{Twists with prescribed critical intervals} In this subsection, we consider twists of $M$ by certain algebraic Hecke characters with prescribed $r$. Let $M$ be a polarized, regular, special polarization over $K$ with coefficients in $E$, of rank $n$ and pure of weight $w(M)$. Fix a CM type $\Phi$ for $L/K$, $\varphi\in J_{E}$ and $r\in\{0,\dots,n\}$. If $r>0$, let $(a_{\tau}^{(r,\varphi)})_{\tau\in\Phi}$ be the tuple of integers defined by
\begin{equation}\label{eqn:defnatau} a_{\tau}^{(r,\varphi)}=w(M)-2p_{r}(\sigma,\varphi)+1,\end{equation}
where $\sigma$ is the restriction of $\tau$ to $K$. For $r=0$, let $a^{(0)}$ be any integer such that
\[ a^{(0)}<\min\{w(M)-2p_{1}(\sigma,\varphi)\}_{\sigma\in J_{K},\varphi\in J_{E}}, \]
\begin{equation}\label{eqn:paritya0} a^{(0)}\equiv w(M)+1(2).\end{equation}

For any $\tau$ and $\varphi$, we let $a_{\tau}^{(0,\varphi)}=a^{(0)}$. Note that $w(M)-2p_{1}(\sigma,\varphi)=p_{n}(\sigma,\varphi)-p_{1}(\sigma,\varphi)\leq 0$, and hence $a_{\tau}^{(0,\varphi)}<0$. In all cases, we have that
\[ a_{\tau}^{(r,\varphi)}\equiv a_{\tau'}^{(r,\varphi)}(2)\quad(\tau,\tau'\in\Phi),\]
and thus by Proposition \ref{prop:existechi}, there exists an algebraic Hecke character $\chi^{(r,\varphi)}$ of $L$, of infinity type $(n_{\tau}^{(r,\varphi)})_{\tau\in\Phi}$, such that 
\[ n_{\tau}^{(r,\varphi)}-n_{\bar\tau}^{(r,\varphi)}=a_{\tau}^{(r,\varphi)}\quad(\tau\in\Phi).\]

\begin{lemma}\label{lemma:propschi} \begin{enumerate}[(i)]
\item The algebraic Hecke characters $\chi^{(r,\varphi)}$ are all critical, except for $r=n/2$ when $n$ is even. In this case, $\chi^{(n/2,\varphi)}$ is critical if and only if $p_{\frac{n}{2}}(\sigma,\varphi)\neq p_{\frac{n}{2}+1}(\sigma,\varphi)+1$ for any $\sigma\in J_{K}$.
\item The infinity type satisfies that $n_{\tau}^{(r,\varphi)}>n_{\bar\tau}^{(r,\varphi)}$ for one (or every) $\tau\in\Phi$ if and only if $r\in\{\lfloor\frac{n}{2}\rfloor+1,\dots,n\}$.
\end{enumerate}
\begin{proof} If $r=0$, then $a_{\tau}^{(0,\varphi)}<0$, which implies that $n_{\tau}^{(0,\varphi)}\neq n_{\bar\tau}^{(0,\varphi)}$ for every $\tau\in\Phi$, so $\chi^{(0,\varphi)}$ is critical. From now on assume that $r>0$. Then $n_{\tau}^{(r,\varphi)}-n_{\bar\tau}^{(r,\varphi)}=w(M)-2p_{r}(\sigma,\varphi)+1$ for $\tau\in\Phi$. Suppose that $\chi^{(r,\varphi)}$ is not critical, so that $n_{\tau}^{(r,\varphi)}=n_{\bar\tau}^{(r,\varphi)}$ for some $\tau\in\Phi$. Then
\[ w(M)=2p_{r}(\sigma,\varphi)-1.\]
Since $w(M)=p_{n+1-r}(\sigma,\varphi)+p_{r}(\sigma,\varphi)$, this amounts to say that 
\begin{equation}\label{eqn:pn+1} p_{n+1-r}(\sigma,\varphi)=p_{r}(\sigma,\varphi)-1.\end{equation}
This necessarily implies that $r\neq n$ and
\[ p_{n+1-r}(\sigma,\varphi)=p_{r+1}(\sigma,\varphi).\]
This in turn implies that $n+1-r=r+1$, so that $n$ is even and $r=n/2$. Then (\ref{eqn:pn+1}) implies that $p_{\frac{n}{2}}(\sigma,\varphi)=p_{\frac{n}{2}+1}(\sigma,\varphi)+1$. This proves part (i). 

For part (ii), the condition is equivalent to $a_{\tau}^{(r,\varphi)}>0$ for every $\tau\in\Phi$, which means that $r\neq 0$ and 
\[ p_{n+1-r}(\sigma,\varphi)-p_{r}(\sigma,\varphi)\geq 0 \]
for every $\sigma\in J_{K}$. This means that $r\geq\frac{n+1}{2}$, or, what is the same, $r\geq\lfloor\frac{n}{2}\rfloor+1$.

\end{proof}
\end{lemma}

\begin{lemma}\label{lemma:resr} Let $r\in\{0,\dots,n\}$. If $n$ is even and $r=n/2$, assume that $p_{\frac{n}{2}}(\sigma,\varphi)\neq p_{\frac{n}{2}+1}(\sigma,\varphi)+1$. Then the realization $M(\chi^{(r,\varphi)})$ has critical values, and for every $\sigma\in J_{K}$, 
\[ r_{\sigma,\varphi,1}(\chi^{(r,\varphi)})=\left\{\begin{array}{lll} n-r & \text{if} & 0\leq r\leq\lfloor\frac{n}{2}\rfloor \\ r & \text{if} & \lfloor\frac{n}{2}\rfloor<r\leq n   ,\end{array}\right.\]
where $1\in J_{\Q(\chi^{(r,\varphi)})}$ denotes the given embedding of $\Q(\chi^{(r,\varphi)})$.
\begin{proof} First we need to check the condition that guarantees that $M(\chi^{(r,\varphi)})$ has critical values. Thus, we need to see that for any $\sigma\in J_{K}$, $\psi\in J_{E}$, $\rho\in J_{\Q(\chi^{(r,\varphi)})}$ and $i=1,\dots,n$, $t_{\sigma,\rho}(\chi^{(r,\varphi)})\neq w(M)-2p_{i}(\sigma,\psi)$. But 
\[ t_{\sigma,\rho}(\chi^{(r,\varphi)})=\pm(n^{(r,\varphi)}(\tau,\rho)-n^{(r,\varphi)}(\bar\tau,\rho)),\]
where $\tau\in\Phi$ extends $\sigma$. Suppose first that $r\neq 0$. Then 
\[ n^{(r,\varphi)}(\tau,\rho)-n^{(r,\varphi)}(\bar\tau,\rho)=n^{(r,\varphi)}_{\tilde\rho^{-1}\tau}-n^{(r,\varphi)}_{\tilde\rho^{-1}\bar\tau}=w(M)-2p_{r}(\sigma',\varphi)+1,\]
where $\sigma'\in J_{K}$ is the restriction of $\tilde\rho^{-1}\sigma$ to $K$. Thus, using that $p_{r}(\sigma',\varphi)+p_{n+1-r}(\sigma',\varphi)=w(M)$, 
\[ t_{\sigma,\rho}(\chi^{(r,\varphi)})=\pm(w(M)-2p_{r}(\sigma',\varphi)+1)=w(M)-2p_{r'}(\sigma',\varphi)\pm1,\]
where $r'$ is either $r$ or $n+1-r$, according to whether the sign is $+$ or $-$. This differs from $w(M)$ by an odd integer, so it can never be equal to $w(M)-2p_{i}(\sigma,\psi)$. 

If $r=0$, then
\[ t_{\sigma,\rho}(\chi^{(r,\varphi)})=-a^{(0)},\]
which by our choices is never equal to any of the $w(M)-2p_{i}(\sigma,\varphi)$. This shows that $M(\chi^{(r,\varphi)})$ has critical values. 

Now, if $\tau\in\Phi$ extends $\sigma$ then, by part (ii) of Lemma \ref{lemma:propschi},
\[ t_{\sigma,1}(\chi)=\left\{\begin{array}{lll}n^{(r,\varphi)}_{\tau}-n^{(r,\varphi)}_{\bar\tau} & \text{if} & \lfloor\frac{n}{2}\rfloor<r\leq n\\
							n^{(r,\varphi)}_{\bar\tau}-n^{(r,\varphi)}_{\tau} & \text{if} & 0\leq r\leq\lfloor\frac{n}{2}\rfloor.\end{array}\right.\]
If $\lfloor\frac{n}{2}\rfloor<r\leq n$, then $t_{\sigma,1}(\chi^{(r,\varphi)})=w(M)-2p_{r}(\sigma,\varphi)+1$. Since $p_{r+1}(\sigma,\varphi)+\frac{1}{2}<p_{r}(\sigma,\varphi)$, it follows that
\[ w(M)-2p_{r}(\sigma,\varphi)<t_{\sigma,1}(\chi^{(r,\varphi)})<w(M)-2p_{r+1}(\sigma,\varphi),\]
so that $r_{\sigma,\varphi,1}(\chi^{(r,\varphi)})=r$. If $0\leq r\leq\lfloor\frac{n}{2}\rfloor$, then $t_{\sigma,1}(\chi^{(r,\varphi)})=-w(M)+2p_{r}(\sigma,\varphi)-1=w(M)-2p_{n+1-r}(\sigma,\varphi)-1$, and by a similar reasoning we get that
\[ w(M)-2p_{n-r}(\sigma,\varphi)<t_{\sigma,1}(\chi^{(r,\varphi)})<w(M)-2p_{n-r+1}(\sigma,\varphi),\]
so that $r_{\sigma,\varphi,1}(\chi^{(r,\varphi)})=n-r$. 
\end{proof}
\end{lemma}

\subsection{Formulas for quadratic periods} \label{subsection:quadratic}
In this subsection, we combine the results above to obtain a formula for the quadratic periods $Q_{j,\sigma}$ in terms of quotients of Deligne $\sigma$-periods of various twists $M(\chi)$, with an eye towards future applications. We first introduce some notation. For any $r=0,\dots,n$, and $\varphi\in J_{E}$, we fix $\chi^{(r,\varphi)}$ as above. Given our choices in (\ref{eqn:defnatau}) and (\ref{eqn:paritya0}), we can assume by Proposition \ref{prop:existechi} that all the characters $\chi^{(r,\varphi)}$ have the same weight $w_{0}\equiv w(M)+1(2)$. We let
\begin{equation}\label{eqn:defPchi} P_{\sigma}(\chi^{(r,\varphi)})=(2\pi i)^{-\lceil\frac{n}{2}\rceil w_{0}}G_{\sigma}(\chi^{(r,\varphi)})^{r}a_{\sigma}^{*}(\chi^{(r,\varphi)})Q_{\sigma}(\chi^{(r,\varphi)})^{r-\lceil\frac{n}{2}\rceil}.\end{equation}
Implicit here is the integer $n$ and its decomposition as $n=n^{+}+n^{-}$, depending on $M$. This is an element of $(\Q(\chi)\otimes\C)^{\times}$, well defined up to multiples in $(\Q(\chi)\otimes\sigma(K))^{\times}$. Note that if $1\leq j<\lceil\frac{n}{2}\rceil$, then $\lfloor\frac{n}{2}\rfloor<n-j\leq n-1$.

\begin{prop}\label{prop:Qjota} Let $\varphi\in J_{E}$ and $j\in\Z$ be an integer with $1\leq j<\lceil\frac{n}{2}\rceil$. Then, for any $\sigma\in J_{K}$, we have that
\[  Q_{j,\sigma,\varphi}\sim_{\varphi(E)\Q(\chi)\sigma(K)}
\begin{cases}
 \dfrac{c_{\sigma}^{+}(M(\chi^{(n-1,\varphi)}))_{\varphi,1}}{P_{\sigma}(\chi^{(n-1,\varphi)})_{1}}\delta_{\sigma}(M)_{\varphi}^{-1} & \mbox{if }  j=1\\
\dfrac{c_{\sigma}^{+}(M(\chi^{(n-j,\varphi}))_{\varphi,1}}{c_{\sigma}^{+}(M(\chi^{(n-j+1,\varphi})_{\varphi,1}}\dfrac{P_{\sigma}(\chi^{(n-j+1,\varphi)})_{1}}{P_{\sigma}(\chi^{(n-j,\varphi)})_{1}} & \mbox{if } j\neq 1.\end{cases}\]
\begin{proof}
Let $r=n-j\in\{\lfloor\frac{n}{2}\rfloor+1,\dots,n-1\}$. We apply the formula (\ref{eqn:formppal2}) to $\chi^{(r,\varphi)}$. By Lemma \ref{lemma:resr}, the formula in this case, looking at the coordinate given by the embeddings $\varphi\in J_{E}$ and $1\in J_{\Q(\chi)}$, says that
\begin{equation}\label{eqn:formQj1} c_{\sigma}^{+}(M(\chi^{(r,\varphi)}))_{\varphi,1}\sim\delta_{\sigma}(M)_{\varphi}P_{\sigma}(\chi^{(r,\varphi)})_{1}\prod_{i=1}^{j}Q_{i,\sigma,\varphi}.\end{equation}
Thus the formula for $j=1$ in the proposition is clear. If $j\geq 2$, then we also have
\begin{equation}\label{eqn:formQj2} c_{\sigma}^{+}(M(\chi^{(r+1,\varphi)}))_{\varphi,1}\sim\delta_{\sigma}(M)_{\varphi}P_{\sigma}(\chi^{(r+1,\varphi)})_{1}\prod_{i=1}^{j-1}Q_{i,\sigma,\varphi}.\end{equation}
Hence, the result follows by dividing (\ref{eqn:formQj1}) by (\ref{eqn:formQj2}) (note that all the relevant elements belong to $\C^{\times}$).
\end{proof}
\end{prop}

We define
\[ Q_{j}=\prod_{\sigma\in J_{K}}Q_{j,\sigma}\in(E\otimes\C)^{\times}/(E\otimes K^{\Gal})^{\times} \]
and
\begin{equation}\label{eqn:defPchitotal} P(\chi^{(r,\varphi)})=\prod_{\sigma\in J_{K}}P_{\sigma}(\chi^{(r,\varphi)})\in(\Q(\chi)\otimes\C)^{\times}/(\Q(\chi)\otimes K^{\Gal})^{\times}.\end{equation}

Using formulas (\ref{form:yoshida1}) and (\ref{form:yoshida2}), we obtain the following corollary of Proposition \ref{prop:Qjota}.

\begin{coro}\label{coro:Qjota} Let the notation and assumptions be as above. Then, for each $\varphi\in J_{E}$, we have that
	\[  Q_{j,\varphi}\sim_{\varphi(E)\Q(\chi)K^{\Gal}}
	\begin{cases}
	\dfrac{c^{+}(M(\chi^{(n-1,\varphi)}))_{\varphi,1}}{P(\chi^{(n-1,\varphi)})_{1}}\delta(M)_{\varphi}^{-1} & \mbox{if }  j=1\\
	\dfrac{c^{+}(M(\chi^{(n-j,\varphi}))_{\varphi,1}}{c^{+}(M(\chi^{(n-j+1,\varphi})_{\varphi,1}}\dfrac{P(\chi^{(n-j+1,\varphi)})_{1}}{P(\chi^{(n-j,\varphi)})_{1}} & \mbox{if } j\neq 1.\end{cases}\]
\end{coro}

Fix $\varphi\in J_{E}$ and $r\in\{0,\dots,n\}$. Let $s=n-r$. For the rest of this subsection, we let $\chi=\chi^{(r,\varphi)}$. Recall that $w_{0}$ is the weight of $\chi$. The period $P(\chi)$ of (\ref{eqn:defPchitotal}) can be interpreted in terms of CM periods, as in Subsection 2.6 of \cite{guerbperiods}. We are only interested in the case $r>\lfloor\frac{n}{2}\rfloor$, and we assume that this is the case from now. In particular, $n_{\tau}>n_{\bar\tau}$ for $\tau\in\Phi$. Moreover, we will only write down the formulas after fixing the embedding $1\in J_{\Q(\chi)}$, since $P(\chi)_{1}$ is what appears in Corollary \ref{coro:Qjota}. In \cite{harriskudla} (see also \cite{harrisunitary}), a family of CM periods attached to $\chi$ is defined. As a particular case, there is a period
\[ p(\chi;\Phi)\in\C^{\times},\]
well defined up to multiples in $\Q(\chi)^{\times}$. For each embedding $\rho\in J_{\Q(\chi)}$, let $\tilde\rho$ be an extension of $\rho$ to $\C$. We can define an algebraic Hecke character $\chi^{\rho}$, of infinity type $(n_{\tilde\rho^{-1}\tau})_{\tau\in J_{L}}$, obtained by applying $\rho$ to the values of $\chi$ on $\A_{L,f}^{\times}$. We can also define a CM type $\tilde\rho\Phi=\{\tilde\rho\tau:\tau\in\Phi\}$. All of these are independent of the extension $\tilde\rho$, and thus we get a well defined CM period $p(\chi^{\rho};\rho\Phi)\in\C^{\times}$. We let
\[ \mathbf{p}(\chi;\Phi)=(p(\chi^{\rho};\rho\Phi))_{\rho\in J_{\Q(\chi)}}\in(\Q(\chi)\otimes\C)^{\times}.\]
The following formula is a theorem of Blasius, and we use the formulation that appears as Proposition 1.8.1 of \cite{harrisunitary}, corrected as in the Introduction to \cite{harriscrelle}. Combined with Deligne's conjecture for the motive $[\chi]$, proved by Blasius (\cite{blasiusannals}), we get that if $m$ is a critical integer for $[\chi]$, then 
\begin{equation}\label{form:blasius} c^{+}([\chi](m))\sim_{\Q(\chi)}D_{K}^{1/2}(2\pi i)^{[K:\Q]m}\mathbf{p}(\check\chi;\Phi),\end{equation}
where $\check\chi=\chi^{\iota,-1}$. Here $\iota\in\Gal(L/K)$ is the non-trivial element. 

The following lemma allows us to relate the quadratic periods $Q_{j,\varphi}$ to quotients of Deligne periods and CM periods, via Corollay \ref{coro:Qjota}. Let
\[ G(\chi)=\prod_{\sigma\in J_{K}}G_{\sigma}(\chi)\in\left(\Q(\chi)\otimes\C\right)^{\times}.\]

\begin{lemma}\label{lemma:interpPchiCM} Let the notation and assumptions be as above. Let $t=0$ if there exist even critical integers $m$ for $[\chi]$, and let $t=1$ otherwise. Then
\[ P(\chi)\sim_{\Q(\chi)\otimes K^{\Gal}}(2\pi i)^{-[K:\Q]w_{0}s}G(\chi)^{s}\mathbf{p}(\check\chi;\Phi)^{r-s}\left(\prod_{\tau\in\Phi}e_{\tau}^{r-n^{+}+t}\right).\]
In particular,
\[ P(\chi)_{1}\sim_{\Q(\chi)K^{\Gal}}(2\pi i)^{-[K:\Q]w_{0}s}G(\chi)_{1}^{s}p(\check\chi;\Phi)^{r-s}.\]
\begin{proof} Throughout, we write $\sim$ for $\sim_{\Q(\chi)\otimes K^{\Gal}}$. By definition, we have that
	\begin{equation}\label{form:recallPchi} P(\chi)\sim(2\pi i)^{-\lceil\frac{n}{2}\rceil w_{0}[K:\Q]}G(\chi)^{r}\left(\prod_{\sigma\in J_{K}}a_{\sigma}^{*}(\chi)\right)\left(\prod_{\sigma\in J_{K}}Q_{\sigma}(\chi)^{r-\lceil\frac{n}{2}\rceil}\right),\end{equation}
	where $a_{\sigma}^{*}(\chi)=1$ if $n$ is even and $a_{\sigma}^{*}(\chi)=a_{\sigma}^{\pm}(\chi)$ if $n$ is odd, where $\pm=-$ if $n^{+}>n^{-}$ and $\pm=+$ if $n^{+}<n^{-}$.	
	
	By Lemma 2.4.1 of \cite{guerbperiods}, we have that
	\begin{equation}\label{eqn:igualde} c_{\sigma}^{-}(\chi)\sim_{\Q(\chi)\otimes K,\sigma}e_{\tau}c_{\sigma}^{+}(\chi),\end{equation}
	for $\tau$ any extension of $\sigma$. It follows from (\ref{form:yoshida1}) that 
	\begin{equation}\label{form:prodasigma} \prod_{\sigma\in J_{K}}a_{\sigma}^{\pm}(\chi)\sim(2\pi i)^{[K:\Q]w_{0}}D_{K}^{-1/2}G(\chi)^{-1}c^{+}(\chi)\left(\prod_{\tau\in\Phi}e_{\tau}^{?}\right), \end{equation}
	where $?=1$ if $\pm=+$, and $?=0$ if $\pm=-$. The factor $D_{K}^{-1/2}$ belongs to $K^{\Gal}$, so it can be ignored in the formula. Similarly,
	\begin{equation}\label{form:prodQsigma} \prod_{\sigma\in J_{K}}Q_{\sigma}(\chi)\sim(2\pi i)^{[K:\Q]w_{0}}G(\chi)^{-2}c^{+}(\chi)^{2}\left(\prod_{\tau\in\Phi}e_{\tau}\right).\end{equation}
Since $c^{+}([\chi](m))\sim(2\pi i)^{[K:\Q]m}c^{(-1)^{m}}(\chi)$, we get from (\ref{form:blasius}) and (\ref{eqn:igualde}) that
\begin{equation}\label{eqn:formfinalc+} c^{+}(\chi)\sim_{\Q(\chi)}D_{K}^{1/2}\mathbf{p}(\check\chi;\Phi)\prod_{\tau\in\Phi}e_{\tau}^{t}.\end{equation}
The lemma follows by combining (\ref{form:recallPchi}), (\ref{eqn:igualde}), (\ref{form:prodasigma}), (\ref{form:prodQsigma}) and (\ref{eqn:formfinalc+}).
\end{proof}
\end{lemma}

\begin{rem} The last lemma is true for any critical $\chi$ such that $n_{\tau}>n_{\bar\tau}$ for $\tau\in\Phi$, as long as the definition (\ref{eqn:defPchi}) of $P(\chi)$ uses the same $r$, $s$ and $n^{+}$.
\end{rem}

\subsection{Deligne periods when $r=n$} In this subsection, we obtain a formula for the Deligne period of $M(\chi)(k)$ whenever $\chi$ has $r=n$ and $k$ is a critical integer. We can use the above formulas for $\chi^{(n,\varphi)}$, but it will be more useful to allow more general characters. Thus, suppose that $\chi$ is an algebraic Hecke character, critical of infinity type $(n_{\tau})_{\tau\in J_{L}}$. Suppose that the CM type $\Phi$ and the infinity type of $\chi$ are related by the condition
\[ n_{\tau}>n_{\bar\tau}\quad(\tau\in\Phi).\]
Moreover, fix $\varphi\in J_{E}$, and suppose that
\begin{equation}\label{eqn:difgrande} n_{\tau}-n_{\bar\tau}>\max\{w(M)-2p_{n}(\sigma,\varphi)\}_{\sigma\in J_{K}} \end{equation}
for all $\tau\in\Phi$. 

\begin{rem}\label{rem:indepvarphi} If (\ref{eqn:difgrande}) holds for the embedding $\varphi\in J_{E}$, then it holds for any other $\psi\in J_{E}$ as well. This follows from the following more general fact. Let $N$ be a realization over $K$, with coefficients in $E$, pure of weight $w(N)$. For $\varphi\in J_{E}$, let 
	\[ T(\varphi)=\bigcup_{\sigma\in J_{K}}\{p\in\Z:N_{\sigma}^{pq}(\varphi)\neq0\}.\]
	Then $T(\varphi)=T(\psi)$ for $\varphi,\psi\in J_{E}$. Indeed, let $p\in T(\varphi)$. Then $\gr^{p}(N_{\dR})\otimes_{E\otimes K,\varphi\otimes\sigma}\C\neq 0$. There exists an element $h\in\Aut(\C)$ such that $\psi=h\varphi$, and then
	\[ 0\neq\gr^{p}(N_{\dR})\otimes_{E\otimes K,\varphi\otimes\sigma}\C\otimes_{\C,h}\C=\gr^{p}(N_{\dR})\otimes_{E\otimes K,\psi\otimes h\sigma},\]
	so $p\in T(\psi)$. 
	Then, (\ref{eqn:difgrande}) can be stated as
	\[ n_{\tau} -n_{\bar\tau}>\max\{w(M)-2p\}_{p\in T(\varphi)},\]
	which is independent of $\varphi$.
\end{rem}

We then have, just as in the proof of Lemma \ref{lemma:resr}, that $M(\chi)$ has critical values and
\begin{equation}\label{eqn:resn} r_{\sigma,\varphi,1}(\chi)=n \end{equation}
for all $\sigma\in J_{K}$ and $\varphi\in J_{E}$. As in (\ref{eqn:defPchi}) and (\ref{eqn:defPchitotal}), we let
\[ P_{\sigma}(\chi)=(2\pi i)^{-\lceil\frac{n}{2}\rceil w(\chi)}G_{\sigma}(\chi)^{n}a_{\sigma}^{*}(\chi)Q_{\sigma}(\chi)^{\lfloor\frac{n}{2}\rfloor} \]
and
\[ P(\chi)=\prod_{\sigma\in J_{K}}P_{\sigma}(\chi),\]
where again $n=n^{+}+n^{-}$ is implicit in the notation. 

Taking into account (\ref{eqn:resn}), formula (\ref{eqn:formppal}) applied to the case of $\chi$ says that
\[ c_{\sigma}^{+}(M(\chi))_{\varphi,1}\sim_{\varphi(E)\Q(\chi)\sigma(K)}\delta_{\sigma}(M)_{\varphi}P_{\sigma}(\chi)_{1}.\]
Using formulas (\ref{form:yoshida1}) and (\ref{form:yoshida2}), we get
\begin{equation}\label{form:c+Mchi} c^{+}(M(\chi))_{\varphi,1}\sim_{\varphi(E)\Q(\chi)K^{\Gal}}\delta(M)_{\varphi}P(\chi)_{1}.\end{equation}

The set of critical integers for $M(\chi)$ is computed in (2.5.2) of \cite{guerbperiods}. In this case, this set consist of those integers $k$ such that
\begin{equation}\label{eqn:criticalinteger} p_{1}(\sigma,\varphi)+n_{\bar\tau}<k\leq p_{n}(\sigma,\varphi)+n_{\tau}\end{equation}
for every $\tau\in\Phi$ (with $\sigma=\tau|_{K}$) and a fixed $\varphi\in J_{E}$. 

\begin{rem} By Remark \ref{rem:indepvarphi}, these inequalities are independent of the chosen $\varphi$. For this we need to assume, as we do, that $E(\chi)$ is a field. In the general case, $M(\chi)$ is a collection of realizations with coefficients in the fields appearing in $E(\chi)$, and a critical integer is defined to be one which is critical for each of these realizations.
\end{rem}
	
	Let $k$ be an integer satisfying (\ref{eqn:criticalinteger}). By Lemma 2.4.1 of \cite{guerbperiods}, 
\[ c_{\sigma}^{+}(M(\chi)(k))\sim_{E(\chi)\otimes K,\sigma}(2\pi i)^{kn}c_{\sigma}^{+}(M(\chi))e_{\tau}^{(-1)^{k}}\]
with $\tau\in\Phi$ extending $\sigma$. Using (\ref{form:yoshida1}), we get that
\[ c^{+}(M(\chi)(k))\sim_{E(\chi)\otimes K^{\Gal}}(2\pi i)^{[K:\Q]kn}c^{+}(M(\chi))\prod_{\tau\in\Phi}e_{\tau}^{(-1)^{k}}.\]
Combining this with (\ref{form:c+Mchi}), we obtain that
\[ c^{+}(M(\chi)(k))_{\varphi,1}\sim_{\varphi(E)\Q(\chi)K^{\Gal}}(2\pi i)^{[K:\Q]kn}\delta(M)_{\varphi}P(\chi)_{1}.\]
Finally, the following result follows from this and Lemma \ref{lemma:interpPchiCM} (see also (2.6.2) of \cite{guerbperiods}). For clarity, we recall all the relevant hypotheses.

\begin{prop}\label{prop:exprc+k} Let $M$ be a regular, special, polarized realization over $K$ with coefficients in $E$, pure of weight $w(M)$ and rank $n$. Let $\Phi$ be a CM type for $L/K$, and let $\chi$ be a critical algebraic Hecke character of $L$ of infinity type $(n_{\tau})_{\tau\in J_{L}}$. Suppose that $n_{\tau}>n_{\bar\tau}$ for $\tau\in\Phi$, and that (\ref{eqn:difgrande}) holds. Let $k$ be an integer satisfying (\ref{eqn:criticalinteger}). Then
\[ c^{+}(M(\chi)(k))_{\varphi,1}\sim_{\varphi(E)\Q(\chi)K^{\Gal}}(2\pi i)^{[K:\Q]kn}\delta(M)_{\varphi}p(\check\chi;\Phi)^{n} \]
for any $\varphi\in J_{E}$.
\end{prop}

\begin{rem}\label{rem:delta} The factor $\delta(M)_{\varphi}$ in the previous proposition needs to be dealt with. It can easily be replaced with a suitable power of $2\pi i$ if $w(M)$ is odd (and hence $n$ is even and the polarization is alternated). More precisely, it can be shown (see Lemma 1.4.12 of \cite{harriscrelle} or Remark 2.3.1 of \cite{guerbperiods}) that $\delta_{\sigma}(M)\sim_{E\otimes K,\sigma}(2\pi i)^{-w(M)n/2}$. Thus, (\ref{form:yoshida2}) implies that
	\[ \delta(M)\sim_{E\otimes K^{\Gal}}(2\pi i)^{-w(M)[K:\Q]n/2} \]
	in this case. When $w(M)$ is even, there is apparently no simple way to obtain such an expression. When comparing with automorphic motives in the following sections, we will deal with this case assuming an extra conjecture. 
\end{rem}

\section{Critical values of automorphic $L$-functions}\label{sec:automorphic}
In this section we recall the main results of \cite{guerbperiods} regarding the critical values of $L$-functions of cohomological automorphic forms on unitary groups. In this paper, we will only care about totally definite unitary groups. 

\subsection{Totally definite unitary groups} Let $L/K$ be a CM extension and $\Phi$ a CM type for $L/K$. Let $V$ be a finite-dimensional $L$-vector space, and $h:V\times V\to L$ be a non-degenerate hermitian form, relative to the non-trivial automorphism $\iota\in\Gal(L/K)$. Let $n=\dim_{L}V$. We let $G$ be the similitude unitary group, with similitude factors in $\Q$, attached to $(V,h)$. Thus, for a $\Q$-algebra $R$, the points of $G$ with values in $R$ are given by
\[ G(R)=\{g\in\Aut_{L\otimes R}(V\otimes R):h_{R}(gu,gv)=\nu(g)h_{R}(u,v)\quad\forall u,v\in V\otimes R\},\]
where $\nu(g)\in R^{\times}$. Here $h_{R}:V\otimes R\times V\otimes R\to L\otimes R$ is given by $h_{R}(u\otimes a,v\otimes b)=h(u,v)\otimes ab$. For each $\tau\in J_{L}$, let $V_{\tau}=V\otimes_{L,\tau}\C$. This is equipped with a hermitian form $h_{\tau}$ relative to complex conjugation on $\C/\R$. In particular, there is a well-defined signature $(r_{\tau},s_{\tau})$. We will assume throughout the paper that $V$ is totally definite. This means that for any $\tau\in\Phi$, the signature is $(r_{\tau},s_{\tau})=(n,0)$. We also fix an $L$-basis $\beta=\{v_{1},\dots,v_{n}\}$ of $V$, orthogonal for $h$. As in (3.1.1) of \cite{guerbperiods}, we can write
\[ G_{\R}\cong\left(\prod_{\tau\in\Phi}GU(n,0)\right)',\quad G_{\C}\cong\left(\prod_{\tau\in\Phi}\GL_{n,\C}\right)\times\GL_{1,\C},\]
where the symbol $'$ means that we are looking at tuples where all the elements have the same multiplier $\nu$. Here, the group $GU(n,0)$ is the usual similitude unitary group over $\R$ of the identity matrix $I_{n}$. There is a maximal torus $T\subset G$, namely the subgroup of automorphism which are diagonal with respect to the basis $\beta$, such that $T_{\C}$ corresponds to the subgroup of diagonal matrices under the second isomorphism above. We let $B\subset G_{\C}$ be the Borel subgroup corresponding to $\left(\prod_{\tau\in\Phi}B_{n,\C}\right)\times\GL_{1,\C}$, where $B_{n,\C}$ is the group of upper triangular matrices in $\GL_{n,\C}$. 

We use the notation of Subsection 3.3 of \cite{guerbperiods} regarding roots and weights. In particular, we identify the group $\Lambda=X^{*}(T)$ with the group of tuples
\[ \mu=\left((a_{\tau,1},\dots,a_{\tau,n})_{\tau\in\Phi};a_{0}\right)\in\left(\prod_{\tau\in\Phi}\Z^{n}\right)\times\Z.\]
The set of dominant weights $\Lambda^{+}$, with respect to the Borel subgroup $B$, are those $\mu$ for which $a_{\tau,1}\geq\dots \geq a_{\tau,n}$ for all $\tau\in\Phi$. In the notation of {\em op. cit.}, the group $K_{x}$ is the whole group $G_{\R}$, $\Lambda^{+}_{x,c}=\Lambda^{+}$ and $w_{0}^{1}=1$.

\subsection{Automorphic forms} A Shimura datum $(G,X)$ is constructed in Subsection 3.2 of \cite{guerbperiods}, assuming that $V$ is not totally definite. In our case, we can still define zero-dimensional varieties $S_{U}$ for a compact open subgroup $U\subset G(\A_{f})$. These are algebraic varieties over the reflex field $E$, which is the field generated over $\Q$ by the elements $\sum_{\tau\in\Phi}\tau(b)$, for $b\in L$. In particular, $E\subset L^{\Gal}$, the Galois closure of $L$ in $\bar\Q$. The set of complex points of $S_{U}$ is the finite set
\[ S_{U}(\C)=G(\Q)\backslash G(\A_{f})/U.\]
Most of what is contained in Section 3 of \cite{guerbperiods} also applies to these zero-dimensional varieties. We denote by $S$ the projective limit of the $S_{U}$.

From now on, fix $\mu\in\Lambda^{+}$ such that the corresponding representation $W=W_{\mu}$ of $G_{\C}$ is defined over $\Q$. This implies that 
\[ a_{\tau,i}=-a_{\tau,n+1-i} \]
for every $\tau\in\Phi$ and $i=1,\dots,n$. Let $\xi=2a_{0}$. We let $\Coh_{G,\mu}$ be the set of cuspidal automorphic representations $\pi$ of $G(\A)$ which are essentially tempered and cohomological of type $\mu$. The last condition means that 
\[ (\pi_{\infty}\otimes_{\C}W_{\mu})^{G(\R)}\neq 0.\]
Let $\pi\in\Coh_{G,\mu}$. The motivic normalization of the standard $L$-function is given by
\[ L^{\mot,S}(s,\pi,\St)=L^{S}\left(s-\frac{n-1}{2},\pi,\St\right),\]
where $\St$ stands for the $L$-function corresponding to the standard representation of the $L$-group of $G$, and $S$ is a big enough finite set of places of $L$, included to ensure that the local base change from $G$ (rather, the unitary group) to $\GL_{n}$ is defined at places outside $S$. We let $E(\pi)$ be a CM field containig $L^{\Gal}$ over which $\pi_{f}$ can be realized. Such a CM field always exist (see \cite{bhr}, Theorem 4.4.1, and \cite{harriscrelle}, 2.6).  We let $\pi_{f,0}$ be a model of $\pi_{f}$ over $E(\pi)$. We will assume from now on the following list of hypotheses.

\begin{hyps}\label{hyp:running} The representation $\pi\in\Coh_{G,\mu}$ satisfies:
	
	\begin{enumerate}
		\item $\pi^{\vee}\cong\pi\otimes\|\nu\|^{\xi}$,
		\item for any $\sigma\in J_{E(\pi)}$, $\pi_{f}^{\sigma}=\pi_{f,0}\otimes_{E(\pi),\sigma}\C$ is essentially tempered, and
		\item $\dim_{\C}\Hom_{\C[G(\A_{f})]}(\pi_{f}^{\sigma},H^{0}(S_{\C},\mathcal{E}_{\mu}))\leq 1$.
	\end{enumerate}
\end{hyps}
Hypothesis \ref{hyp:running}, (1), is assumed for simplifying purposes, and it will be satisfied in the applications of Section \ref{sec:deligne}. Hypotheses \ref{hyp:running}, (2) and (3) are expected to be satisfied in most cases in our applications, and we assume them without comment. In (3), $\mathcal{E}_{\mu}$ is the automorphic vector bundle over the Shimura variety $S$ defined by the representation $W_{\mu}$. 

Automorphic quadratic periods for $\pi$ are defined in Subsection 3.10 of \cite{guerbperiods}. Under our running hypotheses, we can define a holomorphic quadratic period
\[ Q^{\hol}(\pi)\in E(\pi)\otimes\C.\]
We don't need to recall the precise definition of it, but rather its interpretation as a Petersson norm. As in Remarks 3.9.2 and 3.10.1 of \cite{guerbperiods},
\[ Q^{\hol}(\pi)\sim_{E(\pi)\otimes L^{\Gal}}\int_{G(\Q)Z(\A)\backslash G(\A)}f(g)\bar f(g)\|\nu(g)\|^{\xi}dg,\]
where $f$ is an automorphic form on $G(\A)$, contributing to $\pi$ rationally in the sense of the canonical model of $\mathcal{E}_{\mu}$ over $L^{\Gal}$. 

\subsection{The main formula for critical values} Here we recall the main result (Theorem 4.5.1) of \cite{guerbperiods} in the case of totally definite unitary grops. As above, we let $\pi\in\Coh_{G,\mu}$, with $W_{\mu}$ defined over $\Q$, be an automorphic representation satisfying Hypotheses \ref{hyp:running}. Let $\psi$ be an algebraic Hecke character of $L$, of infinity type $(m_{\tau})_{\tau\in J_{L}}$, and let $m>n$ be an integer satisfying 
\begin{equation}\label{eqn:ineqm} m\leq a_{\tau,n}+m_{\tau}-m_{\bar\tau} \end{equation}
for every $\tau\in\Phi$ (this is inequality (4.2.1) of \cite{guerbperiods}). Note that $n$ needs to satisfy the same inequality for such an integer $m$ to exist, which poses the condition that $m_{\tau}-m_{\bar\tau}>n-a_{\tau,n}$ for every $\tau\in\Phi$. For the time being, let $m>n$ be any integer satisfying (\ref{eqn:ineqm}). The representation $\pi\otimes(\psi\circ\det)$ will be denoted by $\pi\otimes\psi$. We let
\[ \tilde\psi=\frac{\psi}{\psi^{\iota}}.\]
In the context of the present paper (totally definite unitary groups) Theorem 4.5.1 of \cite{guerbperiods} is basically the following result. 

\begin{prop}\label{prop:maincritical} Let the notation and assumptions be as above. Then
\[	L^{S,\mot}(m,\pi\otimes\psi,\St)\sim_{E(\pi)\Q(\psi)}(2\pi i)^{[K: \Q]\left(mn-\frac{n(n-1)}{2}\right)-\xi}p(\tilde\psi;\Phi)^{n}.   \]
	
	\begin{proof}  We apply the statement of the main theorem of \cite{guerbperiods} given as formula (4.5.2) of {\em op. cit.}, and ignoring $D_{K}^{n/2}$, since it already belongs to $K^{\Gal}\subset L^{\Gal}\subset E(\pi)$. We also fix the embeddings $1\in J_{E(\pi)}$ and $1\in J_{\Q(\psi)}$. The formula then is
\begin{equation}\label{eqn:formulacritica} L^{S,\mot}(m,\pi\otimes\psi,\St)\sim_{E(\pi)\Q(\psi)}(2\pi i)^{[K: \Q]\left(mn-\frac{n(n-1)}{2}\right)-\xi}Q^{\hol}(\pi)_{1}p(\psi;h)p(\psi^{-1},\bar h). \end{equation}
Here $\mathbb{S}=\Res_{\C/\R}\mathbb{G}_{m,\C}$, where $\mathbb{G}_{m}$ is the multiplicative group, and $h:\mathbb{S}\to(\Res_{L/\Q}\mathbb{G}_{m,L})_{\R}\cong\prod_{\tau\in\Phi}\mathbb{S}$ is given as $h=(h_{\tau})_{\tau\in\Phi}$, with $h_{\tau}:\mathbb{S}\to\mathbb{S}$ defined by $h_{\tau}(z)=z^{n}$ (see \cite{guerbperiods}, 3.10). The map $\bar h$ is given by $\bar h(z)=h(\bar z)$. The elements $p(\psi;h)$ and $p(\psi^{-1};\bar h)$ are CM periods satisfying
\[ p(\psi;h)\sim_{\Q(\psi)L^{\Gal}}\prod_{\tau\in\Phi}p(\psi^{n};\{\tau\}) \]
and
\[ p(\psi^{-1};\bar h)\sim_{\Q(\psi)L^{\Gal}}\prod_{\tau\in\Phi}p(\psi^{\iota,-n};\{\tau\}) \]
(see {\em op. cit.}, (3.10.3)). By Proposition 1.4 and Corollary 1.5 of \cite{harrisunitary}, we can write
\begin{equation}\label{eqn:formppsi} p(\psi;h)p(\psi^{-1};\bar h)\sim_{\Q(\psi)L^{\Gal}}p(\tilde\psi;\Phi)^{n}.\end{equation}

Finally, since the hermitian space $V$ is totally definite, the quadratic period $Q^{\hol}(\pi)$ can be taken to be in $E(\pi)$ (see for instance \cite{harrisquebec}, Section 5). Thus, it can taken to be $1$ in (\ref{eqn:formulacritica}), which, together with (\ref{eqn:formppsi}), proves the formula in the statement of the theorem. 
\end{proof}
\end{prop}

\section{The main theorems}\label{sec:deligne} In this section, we prove a version of Deligne's conjecture for certain twists $M(\chi)$ of realizations $M$ which are potentially automorphic in the sense that, after extending the scalars to a totally real Galois extension $K'/K$, they look like the motives (conjecturally) attached to automorphic representations of $\GL_{n}(\A_{K})$. We relate this, via base change and descent, to automorphic representations of unitary groups, and apply the results of Section \ref{sec:automorphic} to express the critical values of the corresponding $L$-functions. We then compare this expression with the one obtained in Section \ref{sec:motivic} to deduce Deligne's conjecture for automorphic realizations (Theorem \ref{thm:main}), after working over the Galois closure $L^{\Gal}$ and fixing embeddings of the coefficient fields. We then prove the theorem for potentially automorphic realizations (Theorem \ref{potentially}) by means of Brauer's induction and the previous case.

\subsection{Automorphic representations of $\GL_{n}(\A_{K})$}\label{ssec:automrealiz} Fix a totally real field $K$. Let $\Pi$ be a cuspidal automorphic representation of $\GL_{n}(\A_{K})$, satisfying the following properties:
\begin{itemize}
	\item $\Pi^{\vee}\cong\Pi$ (self-duality), and
	\item $\Pi$ is cohomological.	
\end{itemize}
The second condition can be expressed by saying that $\Pi_{\infty}$ has the same infinitesimal character as an irreducible representation of $(\Res_{K/\Q}\GL_{n,K})_{\C}\simeq\prod_{\sigma\in J_{K}}\GL_{n,\C}$. Such an irreducible representation can be parametrized, in the standard way, by a collection of integers $(a_{\sigma,1},\dots,a_{\sigma,n})_{\sigma\in J_{K}}$, called the weight of $\Pi$, with $a_{\sigma,1}\geq\dots\geq a_{\sigma,n}$ for every $\sigma\in J_{K}$. 

We let $\Q(\Pi_{f})$ be the field of definition of $\Pi_{f}$. By Theorem 3.13 of \cite{clozelmotifs}, $\Q(\Pi_{f})$ is a number field, and $\Pi_{f}$ can be defined over $\Q(\Pi_{f})$. We expect the existence of a motive $M=M(\Pi)$ over $K$, with coefficients in a finite extension $E(\Pi)\subset\bar\Q$ of $\Q(\Pi_{f})$, attached to $\Pi$ (the reason we need to allow non-trivial extensions of $\Q(\Pi_{f})$ is that the associated Galois representations may not be defined over the $\lambda$-adic completions of $\Q(\Pi_{f})$; see 1.1 of \cite{harrisant}). The motive $M$ should have rank $n$, weight $w(M)=n-1$, and it should have the property that, for $v$ outside a finite set of places $S$,
\begin{equation}\label{eqn:equalLv} L_{v}\left(s-\frac{n-1}{2},\Pi^{\varphi}\right)=L_{v}(M,s)_{\varphi},\end{equation}
where $\varphi\in J_{E(\Pi)}$. The $\lambda$-adic realizations of $M$ have been constructed by a number of people (\cite{chl}, \cite{shingalois}, \cite{ch}, \cite{sorensen}). Moreover, $M$ should be polarized (see for instance \cite{bellchen} regarding the sign of the polarization) and regular, and the Hodge numbers are recovered from the weight of $\Pi$ by the following recipe. We fix $1\in J_{E(\Pi)}$, and we have that
\begin{equation}\label{eqn:piai} p_{i}(\sigma,1)=a_{\sigma,i}+n-i \end{equation}
for every $\sigma\in J_{K}$, $i=1,\dots,n$. The other Hodge numbers $p_{i}(\sigma,\varphi)$ are obtained by a similar recipe by conjugating the weight of $\Pi$. We stress that the motives $M(\Pi)$ are conjectural.

\begin{dfn}\label{dfn:automorphic} Let $M$ be a realization of rank $n$ over $K$ with coefficients in $E$. We say that $M$ is {\em automorphic} if it is regular, polarized, pure of weight $n-1$, and there exists a self-dual, cohomological, cuspidal automorphic representation $\Pi$ of $\GL_{n}(\A_{K})$ such that (\ref{eqn:equalLv}) holds for $\varphi=1$ and $v$ outside a finite set of places $S$, and (\ref{eqn:piai}) holds. In this case we also say that $M$ is associated with $\Pi$.
	
	We say that $M$ is {\em potentially automorphic} if it is polarized and there exists a finite, totally real Galois extension $K'/K$ such that $M_{K'}=M\times_{K}K'$ is automorphic.
	
	\end{dfn}

\subsection{Transfer} Let $\Pi$ be a self-dual, cohomological, cuspidal automorphic representation of $\GL_{n}(\A_{K})$. For a totally imaginary quadratic extension $L/K$, let $\Pi_{L}$ denote the base change of $\Pi$ from $\GL_{n}(\A_{K})$ to $\GL_{n}(\A_{L})$ (see Theorems 4.2 and 5.1 of \cite{arthurclozel}). If $\Pi\not\cong\Pi\otimes\varepsilon_{L/K}$, then $\Pi_{L}$ is cuspidal. This is always the case if $n$ is odd, for instance. In any case, there always exists a totally imaginary quadratic extension $L$ of the form $L=KF$ for a quadratic imaginary field $F$, such that $\Pi_{L}$ is cuspidal (see Section 1 of \cite{clozelpurity}). From now on, we will fix an $L$ such that $\Pi_{L}$ is cuspidal. We also have that $\Pi_{L}$ is a cohomological and $\Pi_{L}^{\vee}\cong\Pi_{L}^{\iota}\cong\Pi_{L}$.

Let $G$ be a unitary group attached to an $n$-dimensional totally definite hermitian space $V$ over $L/K$, as in Section \ref{sec:automorphic}. We expect the existence of a descent $\pi$ of $\Pi_{L}$, from $\GL_{n}(\A_{L})$ to $G(\A)$. Actually, $\Pi_{L}$ should descent to an $L$-packet of representations of $G(\A)$, but for our purposes, we will just pick one of its members. In a significant number of cases, this has already been proved (\cite{labesse}; see also \cite{mok} and \cite{kmsw}).  For any $\tau\in J_{L}$, let
\[ a_{\tau,i}=a_{\sigma,i}, \]
where $\sigma=\tau|_{K}$, and let
\[ \mu=((a_{\tau,1},\dots,a_{\tau,n})_{\tau\in\Phi};0).\]
We say that $\pi$ is a {\em good descent of} $\Pi_{L}$ (or of $\Pi$) if it is cuspidal, cohomological of type $\mu$, and satisfies Hypotheses \ref{hyp:running} (with $W_{\mu}$ defined over $\Q$). The first hypothesis, \ref{hyp:running} (1), is easy to verify in this case (see \cite{guerbperiods}, Remark 5.2.1). The other two hypotheses are expected to hold, so that good descents are expected to exist. The condition that $W_{\mu}$ is defined over $\Q$ is included in \cite{guerbperiods} for simplicity of notation, and it shouldn't be hard to remove.

Let $\psi$ be an algebraic Hecke character of $L$, of infinity type $(m_{\tau})_{\tau\in J_{L}}$ and weight $w=w(\psi)$. Write
\[ \psi|_{\A_{K}^{\times}}=\psi_{0}\|\cdot\|^{-w} \]
as before, with $\psi_{0}$ of finite order. Define
\[ \chi=\psi^{2}(\psi_{0}\circ N_{L/K})^{-1}.\]
Suppose that $\pi$ is a descent of $\Pi$ to $G$, and that $M$ is an automorphic realization associated with $\Pi$. Then, for a certain finite set of places $S$, we have that
\[ L^{\mot,S} (s-w,\pi\otimes\psi,\St)=L^{S}(M(\chi),s)_{1},\]
where $1\in J_{E}$ stands for the given embedding of $E=E(\Pi)$ into $\C$ (see \cite{harriscrelle}, (3.5.2)).

\subsection{Deligne's conjecture: the automorphic case} Keep the assumptions and notation of the last subsections. In particular, $M$ is an automorphic realization associated with $\Pi$. Assume that $\psi$ is critical, and let $\Phi$ be the CM type defined by the condition
\[ m_{\tau}>m_{\bar\tau}\quad(\tau\in\Phi).\]
Furthermore, assume that
\begin{equation}\label{eqn:condition} m_{\tau}-m_{\bar\tau}>\max\{n-p_{n}(\sigma,1)\}_{\sigma\in J_{K}} \end{equation}
for any $\tau\in\Phi$. Let $(n_{\tau})_{\tau\in\Phi}$ be the infinity type of $\chi$. Then $n_{\tau}=2m_{\tau}$, and thus it satisfies equation (\ref{eqn:difgrande}) for $M$. Then, $M(\chi)$ has critical values, and the critical integers are determined by the inequalities (\ref{eqn:criticalinteger}), which become
\begin{equation}\label{eqn:enteroscriticos} a_{\sigma,1}+n-1+2m_{\bar\tau}<k\leq a_{\sigma,n}+2m_{\tau}. \end{equation}
Note that the condition (\ref{eqn:condition}) imlpies that there always exists at least one $k$ satisfying (\ref{eqn:enteroscriticos}) and $k>w+n$.
	
	At some point we need to deal with the factor $\delta(M)$. This is relatively easy to do when $n$ is even (see Remark \ref{rem:delta}), but for the moment we need the conclusion of this remark as a hypothesis when $n$ is odd. It can be proved assuming a much stronger hypothesis, namely Tate's conjecture for the realization $M$ (see \cite{guerbperiods}, 5.4).
	
	\begin{hyp}\label{hypo:delta} If $n$ is odd, then $\delta(M)_{1}\sim_{EK^{\Gal}}(2\pi i)^{-[K: \Q]\frac{n(n-1)}{2}}$.
		\end{hyp}

The following theorem is our main result in the case of automorphic realizations. We recall all the relevant hypotheses. We stress that assumption (3) below is expected to be satisfied in general, and most of what it involves is already proved in many cases. We also stress that given $M$, there always exist algebraic Hecke characters $\psi$ and integers $k$ as in the statement of the theorem.

\begin{thm}\label{thm:main} Let $M$ be an automorphic realization. Let $\psi$ be a critical algebraic Hecke character of $L$, of infinity type $(m_{\tau})_{\tau\in J_{L}}$ and weight $w$, let $\Phi$ be the CM type defined by the condition $m_{\tau}>m_{\bar\tau}$ for $\tau\in\Phi$, and let 
\[ \chi=\psi^{2}(\psi_{0}\circ N_{L/K})^{-1}.\]
Assume that
\begin{enumerate}
\item either $n$ is even, or $n$ is odd and Hypothesis \ref{hypo:delta} is satisfied,
\item $m_{\tau}-m_{\bar\tau}>\max\{n-p_{n}(\sigma,1)\}_{\sigma\in J_{K}}$ for any $\tau\in\Phi$, and
\item the automorphic representation $\Pi$ giving rise to $M$ has a good descent to a totally definite unitary group over $L/K$.
\end{enumerate}
Then Conjecture \ref{conj:weakdeligne} (the weak form of Deligne's conjecture up to $\Q(\psi)L^{\Gal}$-factors for $1\in J_{E(\chi)}$) is true for all critical integers $k>w+n$ of $M(\chi)$. That is, for such integers $k$, we have
\[ \frac{L(M(\chi),k)_{1}}{c^{+}(M(\chi)(k))_{1}}\in{(E\Q(\psi)L^{\Gal})^{\times}}. \]
	\end{thm}
\begin{proof}
Recall that we defined
\[ \check\chi=\chi^{\iota,-1},\quad\tilde\psi=\frac{\psi}{\psi^{\iota}},\]
where $\iota\in\Gal(L/K)$ is the non-trivial element. Let $G$ be a totally definite unitary group as in the hypotheses, and let $\pi$ be a good descent of $\Pi$ to $G$, so that
\[ L^{\mot,S} (s-w,\pi\otimes\psi,\St)=L^{S}(M(\chi),s)_{1},\]
where $w=m_{\tau}+m_{\bar\tau}$. The remaning (finite) Euler factors, evaluated at $k\in\Z$, only affect this equation up to a multiple in the compositum $(E\Q(\chi))^{\times}\subset(E\Q(\psi))^{\times}$, so we may write
\[ L(M(\chi),k)_{1}\sim_{E\Q(\psi)}L^{\mot,S}(m,\pi\otimes\psi,\St),\]
where $m=k-w$. Our hypotheses imply that $m>n$ and 
\[ m\leq a_{\tau,n}+m_{\tau}-m_{\bar\tau} \]
for any $\tau\in\Phi$. This follows directly from the fact that $k$ is a critical integer of $M(\chi)$ and $w=m_{\tau}+m_{\bar\tau}$, so that 
\[ k\leq a_{\sigma,n}+2m_{\tau} \]
by (\ref{eqn:enteroscriticos}). Thus, all the hypotheses of Proposition \ref{prop:maincritical} are satisfied with the integer $m$, and we can write
\begin{equation}\label{exprautomorfa} L^{\mot,S}(m,\pi\otimes\psi,\St)\sim_{E\Q(\psi)L^{\Gal}}(2\pi i)^{[K: \Q]\left((k-w)n-\frac{n(n-1)}{2}\right)}p(\tilde\psi;\Phi)^{n} \end{equation}
(note that $a_{0}=0$ and thus $\xi=2a_{0}=0$ in this situation, and $E(\pi)$ was taken to contain $L^{\Gal}$ in Proposition \ref{prop:maincritical}). 

Now, note that (\ref{eqn:difgrande}) is satisfied for $\chi$, and hence, by Proposition \ref{prop:exprc+k}, we have that
\begin{equation}\label{primexprc+} c^{+}(M(\chi)(k))_{1}\sim_{E\Q(\chi)K^{\Gal}}(2\pi i)^{[K: \Q]kn}\delta(M)_{1}p(\check\chi;\Phi)^{n}.\end{equation}
By Remark \ref{rem:delta} in the case $n$ even, or by Hypothesis \ref{hypo:delta} in the case $n$ odd, we can write
\begin{equation}\label{eqn:delta} \delta(M)_{1}\sim_{EK^{\Gal}}(2\pi i)^{-[K: \Q]\frac{n(n-1)}{2}}.\end{equation}
Now, note that $\check\chi=\tilde\psi\|\cdot\|^{w}$. Using Proposition 1.4 and Lemma 1.8.3 of \cite{harrisunitary}, we obtain that
\begin{equation}\label{exprCM} p(\check\chi;\Phi)\sim_{\Q(\psi)}p(\tilde\psi;\Phi)(2\pi i)^{-[K: \Q]w}. \end{equation}
It follows by combining (\ref{primexprc+}), (\ref{eqn:delta}) and (\ref{exprCM}), that
\begin{equation}\label{segexprc+} c^{+}(M(\chi)(k))_{1}\sim_{E\Q(\psi)K^{\Gal}}(2\pi i)^{[K: \Q]\left(nk-nw-\frac{n(n-1)}{2}\right)}p(\tilde\psi;\Phi)^{n}.\end{equation}
This is exactly the right-hand side of (\ref{exprautomorfa}), which proves the theorem.
\end{proof}

\subsection{Deligne's conjecture: the potentially automorphic case} Suppose that $M$ is a realization over $K$, which becomes automorphic over $K'$, where $K'/K$ is a Galois, totally real extension contained in $\bar K$. Let $M'=M_{K'}$, and let $p_{i}^{M'}(\sigma',\varphi)$ for $\sigma'\in J_{K'}$, $\varphi\in J_{E}$ and $i=1,\dots,n$, be the Hodge numbers of $M'$. Then
\begin{equation}\label{eqn:hodge'} p_{i}^{M'}(\sigma',\varphi)=p_{i}(\sigma'|_{K},\varphi).\end{equation}
Let $\Pi'$ be the automorphic representation of $\GL_{n}(\A_{K'})$ such that $M'$ is associated with $\Pi'$. Then $\Pi'$ is cohomological of weight $(a_{\sigma',1},\dots,a_{\sigma',n})_{\sigma'\in J_{K'}}$, where
\[ a_{\sigma',i}=p_{i}(\sigma|_{K},1)+i-n.\]
This follows from (\ref{eqn:piai}) and (\ref{eqn:hodge'}).

By Brauer-Salomon's Theorem (see \cite{curtisreiner}, 15.10), there exists a finite family of intermediate fields $K\subset K_{j}\subset K'$ and integers $n_{j}\in\Z$ such that
\begin{itemize}
	\item each $\Gal(K'/K_{j})$ is solvable, and
	\item we have an isomorphism
 \begin{equation}\label{brauer}
1_{\Gal(K'/K)} \simeq \bigoplus_{j} n_{j}       \Ind_{\Gal(K'/K_{j})}^{\Gal(K'/K)}1_{\Gal(K'/K_{j})}. 
 \end{equation}
\end{itemize}
Let $M_{j}=M_{K_{j}}$. Using the Arthur-Clozel theory of base change developed in \cite{arthurclozel}, we can show that $M_{j}$ is automorphic, associated with a certain cuspidal automorphic representation $\Pi_{j}$ of $\GL_{n}(\A_{K_{j}})$. The proof of this fact uses cyclic base change, strong multiplicity one, and an argument of Harris (a nice explanation of this is given in \cite{clozelsato}, \S1). The representation $\Pi_{j}$ is a descent of $\Pi'$, meaning that its base change $\Pi_{j,K'}$ to $K'$ is isomorphic to $\Pi'$. Moreover, $\Pi_{j}$ is cohomological of weight $(a^{\Pi_{j}}_{\sigma_{j},1},\dots,a^{\Pi_{j}}_{\sigma_{j},n})_{\sigma_{j}\in J_{K_{j}}}$, where
\[ a^{\Pi_{j}}_{\sigma_{j},i}=p_{i}(\sigma_{j}|_{K},1)+ i- n.\]

For $L/K$ a totally imaginary quadratic extension, we let $L_{j}=LK_{j}$ and $L'=LK'$. Then each of the extensions $L_{j}/K_{j}$ and $L'/K'$ is a CM extension. We fix from now on $L$ with the property that $\Pi'_{L'}$ is cuspidal. We let $\tilde L$ be the compositum of the Galois closures $L_{j}^{\Gal}$. We claim that each base change $\Pi_{j,L_{j}}$ is also cuspidal. Indeed, we can suppose that $K'/K$ is cyclic of prime degree. If $\Pi_{j,L_{j}}$ is not cuspidal, then $\Pi_{j}\cong\Pi_{j}\otimes\varepsilon_{L_{j}/K_{j}}$, by Theorem 4.2 of \cite{arthurclozel}. Since the base change of $\Pi_{j}$ (resp. $\varepsilon_{L_{j}/K_{j}}$) to $K'$ is $\Pi'$ (resp. $\varepsilon_{L'/K'}$), this would imply that $\Pi'\cong\Pi'\otimes\varepsilon_{L'/K'}$, which would in turn imply by the same theorem that $\Pi'_{L'}$ is not cuspidal.

We now come to the main result of this paper. 

\begin{thm}\label{potentially} Let $M$ be a potentially automorphic realization over $K$. Let $\psi$ be a critical algebraic Hecke character of $L$ of infinity type $(m_{\tau})_{\tau \in J_{L}}$, let $\Phi$ be the CM type defined by $m_{\tau} > m_{\overline{\tau}}$ for $\tau \in \Phi$, and let 
	\[ \chi=\psi^{2}(\psi_{0}\circ N_{L/K})^{-1}.\]	
Assume that:
	\begin{itemize}
		\item[(1)] either $n$ is even, or $n$ is odd and Hypothesis \ref{hypo:delta} is satisfied for $M$,
		\item[(2)] $m_{\tau}- m_{\overline{\tau}}> \max\left\lbrace n- p_{n}(\sigma, 1)\right\rbrace_{\sigma \in J_{K}}$ for any $\tau \in \Phi$, and
		\item[(3)] for each $j$ the automorphic representation $\Pi_{j}$ has a good descent to a totally definite unitary group over $L_{j}/K_{j}$.
		
	\end{itemize} 
	Then Conjecture \ref{conj:weakdeligne} (the weak form of Deligne's conjecture up to $\Q(\psi)\tilde L$-factors for $1\in J_{E(\chi)}$) is true for all critical integers $k>w+n$ of $M(\chi)$. That is, for such integers $k$, we have
	\[ \frac{L(M(\chi),k)_{1}}{c^{+}(M(\chi)(k))_{1}}\in{(E\Q(\psi)\tilde L)^{\times}}. \]
\end{thm}

\begin{proof}  
	From (\ref{brauer}) we deduce a formal equality between the $\lambda$-realizations of  $M(\chi)$ and $\Res_{K_{j}/K}(M(\chi)_{K_{j}})$, which implies the following equality of $L$-functions:
	\begin{equation} \label{A}
	L( M(\chi), s)_{1}= \prod_{j}L( M(\chi)_{K_{j}}, s)_{1}^{n_{j}}.
	\end{equation}

	We denote by $\psi_{j}$ the Hecke character of $L_{j}$ obtained from $\psi$ by base change. Thus, 
	\[ \psi_{j}= \psi\circ N_{L_{j}/L}.\]
	We define $\chi_{j}$ in the same way as $\chi$ was constructed from $\psi$. Then $\chi_{j}= \chi\circ N_{L_{j}/L}$. Before continuing with the proof of the theorem, we need a lemma. Recall that we use the notation $M(\chi)= M\otimes\Res_{L/ K}[\chi]$, $M_{j}=M_{K_{j}}$, and $M_{j}(\chi_{j})= M_{j}\otimes\Res_{L_{j}/ K_{j}}[\chi_{j}]$. 
	\begin{lemma} \label{cambio base caracter} We have an equality of $L$-functions 
		\[ L( M(\chi)_{K_{j}}, s)_{1}= L(M_{j}(\chi_{j}), s)_{1}.\]

	\begin{proof} First note that $M(\chi)_{K_{j}}= M_{j}\otimes(\Res_{L/ K}[\chi])_{K_{j}}$. Then it is enough to verify that $(\Res_{L/ K}[\chi])_{K_{j}}$ and $\Res_{L_{j}/ K_{j}}[\chi_{j}]$ have the same $\lambda$-adic realizations for each finite place $\lambda$ of $\Q(\chi)$. Denote by $\chi_{\lambda}$ and $\chi_{j, \lambda}$ the $\lambda$-adic realizations of $\chi$ and $\chi_{j}$ respectively. By definition, 
		\[\left((\Res_{L/ K}[\chi])_{K_{j}}\right)_{\lambda}= \Ind_{\Gamma_{L}}^{\Gamma_{K}}(\chi_{\lambda})\mid_{\Gamma_{K_{j}}}\]and
		\[\left(\Res_{L_{j}/ K_{j}}[\chi_{j}]\right)_{\lambda}= \Ind_{\Gamma_{L_{j}}}^{\Gamma_{K_{j}}}(\chi_{j, \lambda}).\]
		The proof of the lemma finishes by considering the following $\Gamma_{K_{j}}$-equivariant isomorphism: 
		\[\Ind_{\Gamma_{L}}^{\Gamma_{K}}(\chi_{\lambda})\mid_{\Gamma_{K_{j}}} \longrightarrow \Ind_{\Gamma_{L_{j}}}^{\Gamma_{K_{j}}}(\chi_{j, \lambda}) \ , \ f \longmapsto f\mid_{\Gamma_{K_{j}}}.
		\]
		\end{proof}
		\end{lemma}
	
Returning to the proof of the theorem, we claim now that the hypotheses of Theorem \ref{thm:main} are satisfied for the motive $M_{j}$, associated with $\Pi_{j}$, the CM extension $L_{j}/K_{j}$, the character $\psi_{j}$ and the integer $k$.

Note that $\psi_{j}$ has infinity type $(m_{\tau}^{\psi_{j}})_{\tau\in J_{L_{j}}}$, where $m_{\tau}^{\psi_{j}}=m_{\tau|_{L}}$. In particular, $\psi_{j}$ is critical. If we let $\Phi_{j}$ be the CM type of $L_{j}/K_{j}$ determined by $m_{\tau}^{\psi_{j}}>m_{\bar\tau}^{\psi_{j}}$ for $\tau\in\Phi_{j}$, then $\Phi_{j}$ consists of those embeddings $\tau\in J_{L_{j}}$ such that $\tau|_{L}\in\Phi$. 

Concerning hypothesis (1) of Theorem \ref{thm:main}, note that $\Res_{K_{j}/K}M_{j}\cong M^{[K_{j}:K]}$, as can be easily checked. It then follows that 
\[ \delta(M_{j})_{1}\sim\delta(M)_{1}^{[K_{j}:K]}.\]
Thus, hypothesis (1) is satisfied for $M_{j}$ if $n$ is odd. 

For hypothesis (2), note that if $\tau\in\Phi_{j}$, then $m_{\tau}^{\psi_{j}}-m_{\bar\tau}^{\psi_{j}}=m_{\tau|_{L}}-m_{\bar\tau|_{L}}$. Since $\tau|_{L}$, assumption (2) for $M$ and $\psi$ says that this is strictly larger than $\max\{n-p_{n}(\sigma,1)\}_{\sigma\in J_{K}}$. Then, using relation (\ref{eqn:hodge'}) with $K_{j}$ instead of $K'$, we have that
\[ m_{\tau}^{\psi_{j}}-m_{\bar\tau}^{\psi_{j}}>\max\{n-p_{n}^{M_{j}}(\sigma,1)\}_{\sigma\in J_{K_{j}}},\quad(\tau\in\Phi_{j}),\]
	where we are denoting the Hodge numbers of $M_{j}$ by $p_{i}^{M_{j}}(\sigma,1)$.
	
	Hypothesis (3) of Theorem \ref{thm:main} for $\Pi_{j}$ is already within our assumptions. Finally, suppose that $k>w+n$ is a critical integer for $M(\chi)$. Recall from (\ref{eqn:criticalinteger}) that being critical means that
	\[ p_{1}(\sigma,1)+2m_{\bar\tau}<k\leq p_{n}(\sigma,1)+2m_{\tau} \]
	for every $\tau\in\Phi$. Then, by (\ref{eqn:hodge'}) and the fact that $m_{\tau}^{\psi_{j}}=m_{\tau|_{L}}$, we have that
	\[ p_{1}^{M_{j}}(\sigma,1)+2m_{\bar\tau}^{\psi_{j}}<k\leq p_{n}^{M_{j}}(\sigma,1)+2m_{\tau}^{\psi_{j}} \]
	for every $\tau\in\Phi_{j}$. Thus, we are under the conditions of Theorem \ref{thm:main}, and using Lemma \ref{cambio base caracter}, we obtain that
	\begin{equation} \label{B}
	L( M(\chi)_{K_{j}}, k)_{1}= L( M_{j}(\chi_{j}), k)_{1} \sim_{E\Q(\psi)L_{j}^{\Gal}} c^{+}(M_{j}(\chi_{j})(k))_{1}.
	\end{equation} 
	
	Now, the triple $(M_{j}, \chi_{j}, k)$ satisfies the hypotheses of Proposition \ref{prop:exprc+k}. Thus,
	\begin{equation}\label{C}
	c^{+}(M_{j}(\chi_{j})(k))_{1} \sim_{E\Q(\psi)K_{j}^{\Gal}} (2\pi i)^{[K_{j}:\Q]kn} \delta(M_{j})_{1}p(\check{\chi}_{j};\Phi_{j})^{n}.
	\end{equation}
	
	From (\ref{A}), (\ref{B}) and (\ref{C}) we obtain: 
	\begin{equation} \label{producto}
	L( M(\chi), k)_{1} \sim_{E\Q(\psi)\tilde L} (2\pi i)^{\sum_{j}n_{j}[K_{j}:\Q]kn} \prod_{j}\delta(M_{j})_{1}^{n_{j}}\left(\prod_{j}p(\check{\chi}_{j}; \Phi_{j})^{n_{j}}\right)^{n}.
	\end{equation}
	From (\ref{brauer}) we deduce that $[K:\Q]= \sum_{j}n_{j}[K_{j}:\Q]$, so we obtain that
	\begin{equation}\label{compatibilidad 2 pi i}
	(2\pi i)^{\sum_{j}n_{j}[K_{j}:\Q]kn}= (2\pi i)^{[K: \Q]kn}
	\end{equation}
	As we noted above, $\delta(M_{j})_{1}\sim\delta(M)_{1}^{[K_{j}:K]}$. Then, from the formula $1= \sum_{j}n_{j}[K_{j}: K]$, we deduce that
	\begin{equation}\label{compatibilidad delta}
	\prod_{j}\delta(M_{j})_{1}^{n_{j}}\sim_{E\tilde L} \delta(M)_{1}.
	\end{equation}
	
Finally, we claim that
	\begin{equation} \label{compatibilidad CM}
	\left(\prod_{j}p(\check{\chi}_{j}; \Phi_{j})^{n_{j}}\right)\sim_{\Q(\chi)\tilde L} p(\check{\chi}; \Phi).
	\end{equation}
	Indeed, all the characters $\chi_{j}$ and $\chi$ are critical, and the set of critical integers is the same for all (this follows for example from (\ref{eqn:criticalinteger}) taking $M=\Q(0)$). Fix $m\in\Z$ such an integer. Then the results of Blasius (\cite{blasiusannals}) ((\ref{form:blasius}) and Deligne's conjecture for the motives $[\chi]$ and $[\chi_{j}]$) imply that
	\[ L(\chi,m)_{1}\sim_{\Q(\psi)\tilde L}(2\pi i)^{[K:\Q]m}p(\check\chi;\Phi) \]
	and
	\[ L(\chi_{j},m)_{1}\sim_{\Q(\psi)\tilde L}(2\pi i)^{[K_{j}:\Q]m}p(\check\chi_{j};\Phi_{j}).\]
	Then (\ref{compatibilidad CM}) follows from formula (\ref{A}) and Lemma \ref{cambio base caracter} for $M=\Q(0)$.
	
	From (\ref{producto}), (\ref{compatibilidad 2 pi i}), (\ref{compatibilidad delta}), (\ref{compatibilidad CM}) and Proposition \ref{prop:exprc+k} for the triple $(M,\chi,k)$, we obtain that
	\[ L( M(\chi), k)_{1} \sim_{E\Q(\psi)\tilde L}(2\pi i)^{[K: \Q]kn}\delta(M)_{1}p(\check\chi; \Phi) \sim_{E\Q(\psi)L'} c^{+}(M(\chi)(k))_{1},\]
which ends the proof of the theorem.
\end{proof}

\section{Remarks about quadratic periods}\label{sec:quadratic}
\subsection{Quadratic periods and critical values} In this subsection, we draw some final remarks about quadratic periods. We plan to apply these ideas in a future project involving $p$-adic interpolation and $p$-adic $L$-functions. Let $M$ be an automorphic realization of rank $n$ over $K$ with coefficients in $E$. In this section we suppose that \emph{$n$ is even}. We will use the same notation as in Section \ref{sec:deligne}. Let $r \in \left\lbrace \frac{n}{2}+ 1,\dots,n- 1 \right\rbrace$ and suppose that:  
\begin{equation}
a_{\sigma,r} \equiv a_{\sigma',r} (2), \ \ \mathrm{for \ all} \ \sigma, \sigma' \in J_{K}.
\end{equation}
This hypothesis is equivalent to $p_{r}(\sigma,1)\equiv p_{r}(\sigma,1) (2)$ for all  $\sigma, \sigma' \in J_{K}$. As $n$ is even and the weight of $M$ is $n-1$, using Proposition \ref{prop:existechi} we can construct an algebraic Hecke character $\psi^{(r,1)}$ of $L$ of infinity type $(m_{\tau})_{\tau \in J_{L}}$, such that
\[ m_{\tau}-m_{\bar\tau}=\frac{n}{2}-p_{r}(\sigma,1)=\frac{n}{2}-a_{\sigma,r}-s\quad(\tau\in\Phi).\]

As before, we let
\[ \chi^{(r,1)}=(\psi^{(r,1)})^{2}(\psi^{(r,1)}_{0}\circ N_{L/K}), \]
so that its infinity type is $(n_{\tau})_{\tau\in J_{L}}$, with $n_{\tau}=2m_{\tau}$ and
\[ n_{\tau}-n_{\bar\tau}=n-2p_{r}(\sigma,1)=n-2s-2a_{\sigma,r}\quad(\tau\in\Phi).\]

Let $\sigma \in J_{K}$ and suppose that
\begin{equation}
a_{\sigma,r+ 1} \equiv a_{\sigma,r}+ 1 (2), \ \ \mathrm{for \ all} \ r \in \left\lbrace \frac{n}{2}+ 1,\dots, n- 1 \right\rbrace.
\end{equation}
This condition allows us to choose $\psi^{(\frac{n}{2}+ 1,1)}, \dots, \psi^{(n-1 ,1)}$ such that each $\chi^{(r,1)}$, for $r= \frac{n}{2}+ 1,\dots, n- 1$, has weight $w_{0}$, independent of $r$, and satisfying:
\begin{itemize}
\item[a)] $w_{0} \equiv n (2)$,
\item[b)] $2r- n$ is a divisor of $w_{0}$ for any $r \in \left\lbrace \frac{n}{2}+ 1,\dots, n- 1 \right\rbrace$, and
\item[c)] $4$ is a divisor of $w_{0}+ n$.
\end{itemize}

\begin{lemma} Let $r \in \left\lbrace \frac{n}{2}+ 1,\dots, n- 1 \right\rbrace$. Then the only integer $m \in \Z$ which is critical for $M(\chi^{(r,1)})$ is $m= \frac{n+ w_{0}}{2}$.
\end{lemma}
\begin{proof} Recall that in this case we have $p_{r}(\sigma, 1)= a_{\sigma, r}+ n- r$ for each $\sigma \in J_{K}$. Now we apply the formulas obtained in Subsection 5.3 of \cite{guerbperiods} which determine the set of critical integers, considering the following facts: i) $a_{\sigma,1} \geq ... \geq a_{\sigma, n}$, ii) $a_{\sigma,r}+ a_{\sigma,n-r+1}= 0$ and iii) $m_{\tau}^{(r,1)}- m_{\overline{\tau}}^{(r,1)}= \frac{n- 2p_{r}(\sigma, 1)}{2}$. \end{proof}

Let $F$ be the number field generated by the image of $\chi^{(r,1)}$ on the finite ad\`eles $\A_{L,f}^{\times}$ for $r= \frac{n}{2}+ 1,\dots, n- 1$ (that is, we take the compositum of the $\Q(\chi^{(r,1)})$ for $r= \frac{n}{2}+ 1,\dots, n- 1$).

\begin{prop} \label{qj formula} Let $j \in \left\lbrace 2,\dots, \frac{n}{2}- 1\right\rbrace $. Suppose that Conjecture \ref{conj:weakdeligne} up to $FK^{\Gal}$-factors for the embedding $1$ is true for $M(\chi^{(n-j, 1)})$ and $M(\chi^{(n-j+ 1, 1)})$ for the critical integer $\frac{n+ w_{0}}{2}$. Then
\[Q_{j,1} \sim_{EFK^{\Gal}}\]
\[\frac{L(M(\chi^{(n-j, 1)}), \frac{n+ w_{0}}{2})_{1}}{L(M(\chi^{(n-j+ 1, 1)}), \frac{n+ w_{0}}{2})_{1}}\frac{L(\chi^{(n-j+ 1, 1)}, -\frac{w_{0}(j-1)}{n- 2(j-1)})^{n- 2(j-1)}}{L(\chi^{(n-j, 1)}, -\frac{w_{0}j}{n- 2j})^{n- 2j}}\frac{G(\chi^{(n-j+ 1, 1)})_{1}^{j- 1}}{G(\chi^{(n-j, 1)})_{1}^{j}}\]
\end{prop}
\begin{proof} First note that for each $r=0,\dots,n$, by Deligne's conjecture for algebraic Hecke characters (proved by Blasius in \cite{blasiusannals}) and Lemma \ref{lemma:interpPchiCM}, we obtain that
\[ P(\chi^{(r,1)})_{1}\sim_{\Q(\chi^{(r,1)})L^{\Gal}} L\left( \chi^{(r,1)}, \dfrac{-w_{0}(n-r)}{2r- n}\right) ^{2r- n} G(\chi)_{1}^{n- r}. \]
Then the proof finishes by using this formula, Corollary \ref{coro:Qjota} and the hypothesis. \end{proof}

\begin{rem} When $j=1$, we deduce an analogous formula to that of Proposition \ref{qj formula}.
\end{rem}

\begin{rem} In Subsection 3.10 of \cite{guerbperiods}, the author defines \emph{automorphic quadratic periods}. A particular case of this is the period denoted by $Q^{\hol}(\pi)$ in Section \ref{sec:automorphic} above, but in arbitrary signatures they can be defined for automorphic forms contributing in coherent cohomology to other non-holomorphic degrees. Proposition \ref{qj formula}, combined with the comparisons in 5.4 of {\em op. cit.}, suggests an expression for these periods in terms of critical values of automorphic $L$-functions. This formula should be useful, for example, in questions on $p$-adic interpolation related with automorphic quadratic periods.
\end{rem}

%\begin{rem} Moreo ver remark that the automorphic method normally need to work to $\frac{n}{2}$ (recall that in \cite{guerbperiods} the critical value need to satisfy $> n$).\end{rem}

\bibliography{criticalpotentially.bib} \bibliographystyle{amsplain}

\end{document}